\documentclass[reqno,12pt]{amsart}
\usepackage{amssymb,amsthm,amsmath,amsfonts,amstext,mathrsfs,fancybox}
\usepackage{latexsym}
\usepackage{epsfig,url}

\usepackage{fancyhdr}
\usepackage[unicode]{hyperref}

\usepackage{pgf}
\usepackage{tikz}

\usepackage{verbatim}

\usepackage[utf8]{inputenc}
\usetikzlibrary{arrows,automata}
\usetikzlibrary{positioning}

\tikzset{
    state/.style={
           rectangle,
           rounded corners,
           draw=black, very thick,
           minimum height=2em,
           inner sep=2pt,
           text centered,
           },
}

\hypersetup{colorlinks=true,
linkcolor=blue,
citecolor=blue,
urlcolor=blue,
filecolor=blue,
bookmarksnumbered=true,pdfstartview=FitH,
pdfhighlight=/N
}

\newtheorem{thm}{Theorem}

\newtheorem{prop}{Proposition}
\newtheorem{defin}{Definition}

\newtheorem*{prob9}{Problem 9}
\newtheorem*{prob10}{Problem 10}

\input xy
\xyoption{all}

\newcommand{\m}{\mathbf}
\newcommand{\bl}{\bullet}

\def\v{\,{\varsigma}}
\def\k{\mathbf{k}}
\def\l{\mathbf{l}}
\def\n{\mathbf{m}}
\def\a{\,{\mathbf{a}}}
\def\b{\,{\mathbf{b}}}
\def\d{\,{\rm{d}}}
\def\sm{\,{\rm{sm}}}
\def\cm{\,{\rm{cm}}}
\def\s{\,{\rm{sp}}}
\def\c{\,{\rm{cp}}}

\title[Abelian projective plane flows]
{Algebraic and abelian solutions to\\ 
the projective translation equation}
\author[G. Alkauskas]{Giedrius Alkauskas}
\address{Vilnius University, Department of Mathematics and Informatics, Naugarduko 24, LT-03225 Vilnius, Lithuania}
\email{giedrius.alkauskas@mif.vu.lt}

\newcounter{noteno}

\newenvironment{Note}
	{\refstepcounter{noteno}
	\begin{small}
	\medbreak\par\noindent{{\bf Note ~\thenoteno}.}}
	{\hfill{$\Box$}\end{small}\par\medbreak}

\begin{document}
\begin{abstract} Let $\m{x}=(x,y)$. A projective two-dimensional flow is a solution to a $2-$dimensional projective translation equation (PrTE)
$(1-z)\phi(\m{x})=\phi(\phi(\m{x}z)(1-z)/z)$, $\phi:\mathbb{C}^{2}\mapsto\mathbb{C}^{2}$. Previously we have found all solutions of the PrTE which are rational functions. The rational flow gives rise to a vector field $\varpi(x,y)\bl \varrho(x,y)$ which is a pair of $2-$homogenic rational functions. On the other hand, only very special pairs of $2-$homogenic rational functions, as vector fields, give rise to rational flows. The main ingredient in the proof of the classifying  theorem is a reduction algorithm for a pair of $2$-homogenic rational functions. This reduction method in fact allows to derive more results. \\
\indent Namely, in this work we find all projective flows with rational vector fields whose orbits are algebraic curves. We call these flows \emph{abelian projective flows}, since either these flows are parametrized by abelian functions and with the help of 1-homogenic birational plane transformations (1-BIR) the orbits of these flows can be transformed into algebraic curves $x^{A}(x-y)^{B}y^{C}\equiv\mathrm{const.}$ (abelian flows of type I), or there exists a $1-$BIR which transforms the orbits into the lines $y\equiv\mathrm{const.}$ (abelian flows of type II), and generally the latter flows are described in terms of non-arithmetic functions.\\ 
\indent Our second result classifies all abelian flows which are given by two variable algebraic functions. We call these flows \emph{algebraic projective flows}, and these are abelian flows of type I. We also provide many examples of algebraic, abelian and non-abelian flows.

\end{abstract}

\pagestyle{fancy}
\fancyhead{}
\fancyhead[LE]{{\sc Abelian and algebraic flows}}
\fancyhead[RO]{{\sc G. Alkauskas}}
\fancyhead[CE,CO]{\thepage}
\fancyfoot{}

\date{January 7, 2016}
\thanks{The research of the author was supported by the Research Council of Lithuania grant No. MIP-072/2015}

\maketitle

\setcounter{tocdepth}{1} 
\tableofcontents

\section{Background and main results}
For typographical reasons, we write $F(x,y)\bl G(x,y)$ instead of
$\big{(}F(x,y),G(x,y)\big{)}$. We also write $\m{x}=x\bl y$. For a general theory of the translation equation, we refer to \cite{aczel,moszner1,moszner2}. The 2-dimensional (affine) flows are treated in \cite{nikolaev}. The relation of projective flows with P\'{o}lya urn theory and integrable urns was shown in \cite{alkauskas}. The reader may consult 
\cite{flajolet-b, flajolet, flajolet-c} for the background on urns and elliptic functions.\\

The \emph{projective translation equation}, or PrTE, was first introduced in \cite{alkauskas-t} and is the equation of the form
\begin{eqnarray*}
(1-z)\phi(\m{x})=\phi\Big{(}\phi(\m{x}z)\frac{1-z}{z}\Big{)}.
\end{eqnarray*}
In a $2-$dimensional case, $\phi(x,y)=u(x,y)\bl v(x,y)$ is a pair of functions in two real or complex variables. The non-singular solution of this equation is called \emph{the projective flow}.  The \emph{non-singularity} means that a flow satisfies the boundary condition
\begin{eqnarray}
\lim\limits_{z\rightarrow 0}\frac{\phi(\m{x}z)}{z}=\m{x}.
\label{bound}
\end{eqnarray}
What follows next is an equivalent version of the PrTE. In fact, these two are equivalent in an unramified case \cite{alkauskas-un}. In a ramified case, which is the topic of the current paper (see further), the following form of the PrTE is unambiguous for fixed $\m{x}$ and sufficiently small $z,w$  and therefore should be used instead:
\begin{eqnarray}\setlength{\shadowsize}{2pt}\shadowbox{$\displaystyle{\quad
\frac{1}{z+w}\,\phi\big{(}\m{x}(z+w)\big{)}=\frac{1}{w}\,\phi\Big{(}\phi(\m{x}z)\frac{w}{z}\Big{)}},\quad w,z\in\mathbb{C}$.\quad}\label{funk}
\end{eqnarray}
The main object associated with a smooth projective flow is its \emph{vector field} given by
\begin{eqnarray}
\mathbf{v}(\phi;x,y)=\varpi(x,y)\bl\varrho(x,y)=\frac{\d}{\d z}\frac{\phi(xz,yz)}{z}\Big{|}_{z=0}.
\label{vec}
\end{eqnarray}
Vector field is necessarily a pair of $2$-homogenic functions. For smooth functions, the functional equation (\ref{funk}) implies the PDE
\begin{eqnarray}
u_{x}(x,y)(\varpi(x,y)-x)+
u_{y}(x,y)(\varrho(x,y)-y)=-u(x,y),\label{g-parteq}
\end{eqnarray}
and the same PDE for $v$, with boundary conditions given by (\ref{bound}); that is,
\begin{eqnarray*}
\lim\limits_{z\rightarrow 0}\frac{u(xz,yz)}{z}=x,\quad
\lim\limits_{z\rightarrow 0}\frac{v(xz,yz)}{z}=y.
\end{eqnarray*}
In principle, these two PDEs (\ref{g-parteq}) with the above boundary conditions are equivalent to (\ref{funk}).\\

Each point under a flow possesses the orbit, which is either a curve (when we deal with flows over $\mathbb{R}$) or a surface (complex curve over $\mathbb{C}$), or a single point. The orbit is defined by
\begin{eqnarray*}
\mathscr{V}(\m{x})=\Big{\{}\frac{\phi(\m{x}z)}{z}:z\in\mathbb{R}\text{ or }\mathbb{C}\Big{\}}.
\end{eqnarray*}

The orbits of the flow with the vector field $\varpi\bl\varrho$ are given by $\mathscr{W}(x,y)=\mathrm{const}.$, where $\mathscr{W}$ is an $N-$homogenic function which can be found from the differential equation
\begin{eqnarray}
N\mathscr{W}(x,y)\varrho(x,y)+\mathscr{W}_{x}(x,y)[y\varpi(x,y)-x\varrho(x,y)]=0.
\label{orbits}
\end{eqnarray}
We call $N$ \emph{the level of the flow}; that is, the degree of a curve if it is algebraic. If not, we put $N=1$.
\begin{defin}We call the flow $\phi$ \emph{an abelian projective flow}, if its vector field is rational, and its orbits are algebraic curves.
\end{defin}
The reason for this name is that it will appear later that for such a flow $\phi(\m{x})=u(x,y)\bl v(x,y)$ in a reduced form, the function $k(z)=\frac{u(xz,yz)}{v(xz,yz)}$ for fixed $x,y$ and varying $z$ will be the inversion of an abelian integral; hence, algebraic or abelian \cite{lang}. Also note that our research has a strong resemblance to the area of mathematics lying in the intersection of the theory of partial differential equations and algebraic geometry \cite{milne,mumford}.
\begin{defin}We call the flow $\phi$ \emph{an algebraic projective flow}, if its vector field is rational, and $u\bl v$ is a pair of algebraic functions.
\end{defin}
We note that, trivially, an algebraic flow is necessarily abelian.\\

The following is the main result of \cite{alkauskas}. Recall that $1-$BIR stands for a $1-$homogenic birational plane transformation.
\begin{thm}
Let $\phi(x,y)=u(x,y)\bl v(x,y)$ be a pair of rational functions in $\mathbb{C}(x,y)$ which satisfies the functional equation (\ref{funk}) and the boundary condition (\ref{bound}). Assume that $\phi(\m{x})\neq \phi_{{\rm id}}(\m{x}):=x\bl y$.
Then there exists an integer $N\geq 0$, which  is the basic invariant of the flow, called \emph{the level}. Such a flow $\phi(x,y)$ can be given by
\begin{eqnarray*}
\phi(x,y)=\ell^{-1}\circ\phi_{N}\circ\ell(x,y),
\end{eqnarray*}
where $\ell$ is a $1-$BIR, and $\phi_{N}$ is the canonical solution of level $N$ given by
\begin{eqnarray*}
\phi_{N}(x,y)=x(y+1)^{N-1}\bl \frac{y}{y+1}.
\label{can}
\end{eqnarray*}
\label{thm1}
\end{thm}
The same result holds for $u,v\in\mathbb{R}(x,y)$.\\

In this paper we generalize this theorem. First, we classify all abelian flows. This is a rather general result, since all the interesting details occur only in particular examples; see Section \ref{abel-non-ell}.  
\begin{thm} Let $\phi(\m{x})$ be an abelian flow, $\phi(\m{x})\neq\phi_{\mathrm{id}}(\m{x})$. Then there are two possibilities. \\

$\mathrm{I.}$ There exists an $1$-BIR $\ell$ such that, if we define
\begin{eqnarray}
\eta(\m{x})=\ell^{-1}\circ\phi\circ\ell(\m{x}),
\label{eta}
\end{eqnarray} 
then $\eta(\m{x})$ is a flow whose vector field is given by
\begin{eqnarray}
\varpi(x,y)\bl\varrho(x,y)=(B-1)x^{2}+B(C-1)xy\bl (B-1)Cxy+(C-1)y^2,
\label{ova}
\end{eqnarray}  
where $B,C\in\mathbb{Q}\setminus\{1\}$, and orbits are algebraic curves
\begin{eqnarray*}
x^{(1-C)M}(x-y)^{(BC-1)M}y^{(1-B)M}\equiv\mathrm{const}.
\end{eqnarray*}
Here $M\in\mathbb{Q}$ is such that all exponents are integers. The vector fields $\varpi\bl\varrho$ with parameters $(B,C)$, $B,C\neq 1$, are linearly conjugate for $(B,C)$ being any pair from the triple
\begin{eqnarray*}
B,\quad C,\quad \frac{B+C-2}{BC-1}.
\end{eqnarray*} 

$\mathrm{II.}$ There exists an $1-$BIR $\ell$ such that if $\eta(\m{x})=u(x,y)\bl v(x,y)$ is again defined by (\ref{eta}), then $v(x,y)=y$. Orbits of the flow $\eta(\m{x})$ are the lines $y=\mathrm{const.}$, and if $\varpi(x,y)\bl 0$ is the vector field of $\eta(\m{x})$, then the numerator of $\varpi(x,y)$ has no multiple roots.
\label{thm2}
\end{thm}
As our second main result, we find all algebraic flows. This result is a generalization of Theorem \ref{thm1} and in fact implies the latter. For a non-negative integer $n$, and a rational number $Q$, 
$\frac{1}{Q}\notin\{-n,-n+1,\ldots,-1\}$, we need the following polynomial
\begin{eqnarray*}
P_{n,Q}(x):=\,_{2}F_{1}\Big{(}-n,1;1+\frac{1}{Q};x\Big{)},
\end{eqnarray*}
where
\begin{eqnarray*}
\,_{2}F_{1}(a,1;c;x)=\sum\limits_{i=0}^{\infty}\frac{(a)_{i}}{(c)_{i}}\,x^{i},\quad 
(q)_{i}=q(q+1)\cdots (q+i-1) \text{ for }i\geq 1,\quad (q)_{0}=1,
^{}\end{eqnarray*}
is the hypergeometric function. Indeed, $P_{n,Q}(x)$ are polynomials of degree $n$:
\begin{eqnarray*}
P_{0,Q}(x)&=&1,\\
P_{1,Q}(x)&=&1-\frac{Q}{Q+1}x,\\
P_{2,Q}(x)&=&1-\frac{2Q}{Q+1}x+\frac{2Q^2}{(Q+1)(2Q+1)}x^2,\\
P_{3,Q}(x)&=&1-\frac{3Q}{Q+1}x+\frac{6Q^2}{(Q+1)(2Q+1)}x^2
-\frac{6Q^3}{(Q+1)(2Q+1)(3Q+1)}x^3.\\
\end{eqnarray*}
\begin{thm} Let $\phi(\m{x})$ be an algebraic flow, $\phi(\m{x})\neq\phi_{\mathrm{id}}(\m{x})$. Then there exists an ordered pair $(n,Q)$, $n\in\mathbb{N}_{0}$, $Q\in\mathbb{Q}$, $\frac{1}{Q}\notin\{-n,-n+1,\ldots,-1\}$, called \emph{the type}, and an 1-BIR $\ell$ such that, if we define $\eta(\m{x})$ by (\ref{eta}), then the vector field of $\eta$ is given by
\begin{eqnarray*}
Q(n+1)x^{2}+(1-Q)xy\bl Qnxy+y^2;
\end{eqnarray*} 
(that is, up to a scalar conjugation, in Theorem \ref{thm2} we have $C=\frac{n}{n+1}$, $B=1-Q$). The flow $\eta(\m{x})=u(x,y)\bl v(x,y)$ itself is given explicitly by (\ref{alg-expli}) and implicitly from two algebraic equations
\begin{eqnarray*}
u^{\frac{1}{Q}}(u-v)^{-n-\frac{1}{Q}}v^{n+1}=x^{\frac{1}{Q}}(x-y)^{-n-\frac{1}{Q}}y^{n+1},\quad
\frac{1}{v}P_{n,Q}\Big{(}\frac{u}{v}\Big{)}=\frac{1}{y}P_{n,Q}\Big{(}\frac{x}{y}\Big{)}-1.
\end{eqnarray*}
The first, raised to a power $\mathrm{numer}(Q)$, is a polynomial equation, and the second is polynomial equation of total degree $2n+2$. 

\label{thm3}
\end{thm}
We note that we can tell precisely when the flow $\phi$ is not only algebraic but in fact rational, since we have at our disposition the criteria (\cite{alkauskas}, Proposition 3). So, this happens only for the type $(n,0)$ (rational flow of level $0$ in the meaning of \cite{alkauskas} and Theorem \ref{thm1}), or type $(0,Q)$, $Q\in\mathbb{Z}$ (rational flow of level $Q$). \\ 

We will frequently need a $1-$BIR which transforms the second coordinate of the vector field into $\pm y^{2}$. This transformation can be written in the form $xA(x,y)\bl yA(x,y)$, where $A(x,y)$ is a $0-$homogenic function, and $A(x,1)=f(x)$ itself is found from the differential equation (see \cite{alkauskas}, Eq. (33))
\begin{eqnarray}
f(x)\varrho(x)+f'(x)(x\varrho(x)-\varpi(x))=\pm 1.
\label{33}
\end{eqnarray}
Here $\varrho(x)=\varrho(x,1)$, and $\varpi(x)=\varpi(x,1)$.

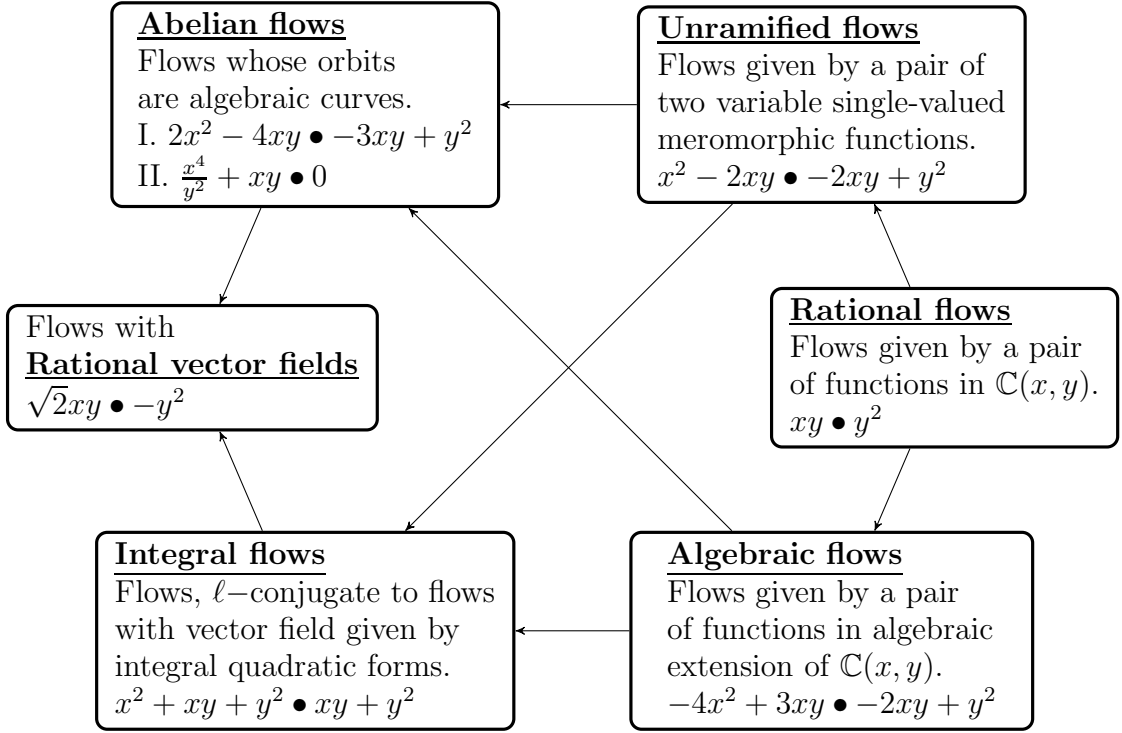
\begin{figure}
\begin{tikzpicture}[->,>=stealth']

 \node[state] (RVF) 
 {\begin{tabular}{l}
 Flows with\\
\underline{\textbf{Rational vector fields}}\\
$\sqrt{2}xy\bl-y^{2}$ 
 \end{tabular}}; 
 
 \node[state,
  yshift=3.5cm,
  xshift=1.5cm,
  anchor=center,
  text width=5cm] (ABF) 
 {
 \begin{tabular}{l}
  \underline{\textbf{Abelian flows}}\\
  Flows whose orbits\\
  are algebraic curves.\\
  I. $2x^2-4xy\bl-3xy+y^{2}$\\
  II. $\frac{x^4}{y^2}+xy\bl 0$
 
 \end{tabular}
 };
 
 \node[state,
 yshift=-3.5cm,
 xshift=1.5cm,] (RVF2) 
 {\begin{tabular}{l}
\underline{\textbf{Integral flows}}\\
Flows, $\ell-$conjugate to flows\\ 
with vector field given by\\
integral quadratic forms.\\ 
 $x^2+xy+y^2\bl xy+y^{2}$ 
 \end{tabular}};
 
 \node[state,
  xshift=8.5cm,
  yshift=3.5cm,
  node distance=5cm,
  anchor=center,
  text width=5cm] (UF) 
 {
 \begin{tabular}{l}
 \underline{\textbf{Unramified flows}}\\
 Flows given by a pair of\\
 two variable single-valued\\
 meromorphic functions.\\
 $x^2-2xy\bl-2xy+y^{2}$  
 \end{tabular}
 };
 
\node[state,
  yshift=0cm,
  xshift=10cm,
  node distance=2cm,
  anchor=center] (RF) 
 {
 \begin{tabular}{l}
 \underline{\textbf{Rational flows}}\\
 Flows given by a pair\\
 of functions in $\mathbb{C}(x,y)$.\\
 $xy\bl y^2$ 
 \end{tabular}
 }; 
  
 \node[state,
  xshift=8.5cm,
  yshift=-3.5cm,
  anchor=center,
  text width=5.2cm] (ALF) 
 {
 \begin{tabular}{l}
  \underline{\textbf{Algebraic flows}}\\
 Flows given by a pair\\
 of functions in algebraic\\ extension of $\mathbb{C}(x,y)$.\\
 $-4x^2+3xy\bl-2xy+y^{2}$
 \end{tabular}
 };

 \path  (ABF)	edge  node[anchor=south,above]{} (RVF)
(UF)    edge                 (ABF)                                     
(RF)    edge                 (UF)         
(UF)    edge (RVF2)
(RF)    edge                 (ALF)                             
(ALF)  	edge                (ABF)
(ALF)  	edge                (RVF2)
(RVF2)  	edge                 (RVF) ;
\end{tikzpicture}
\caption{Diagram of arithmetic classification of projective plane flows. The arrow stands for the proper inclusion. The intersection of unramified and algebraic flows is exactly the set of rational flows. Everywhere we assume that vector field is a pair of rational functions, necessarily $2-$homogenic. Of course, the variety of unramified or algebraic flows whose vector field is non-rational is much bigger. In each category we give an example of the vector field. These examples are treated in detail in the text, and in each case the vector field is not in the image of an inclusion represented by the arrow. Algebraic flows are abelian flows of type I.}
\label{diagram}
\end{figure}

\section{Special examples}
We now proceed with eight examples illustrating the Diagram in Figure \ref{diagram}, and Theorems \ref{thm1}, \ref{thm2}, and \ref{thm3}. In the first four cases the vector field is of the form $ax^2+bxy\bl cxy+y^2$, $a,b,c\in\mathbb{Z}$, $a+b=c+1$. This guarantees that the orbits are given by the curves $x^{A}(x-y)^{B}y^{C}=\mathrm{const.}$, and at least on the lines $x=0$, $y=0$, $x=y$ the flow reduces to a one-dimensional flow, and so can be easily given a simple expression.\\

First, a trivial example just for completeness.
\begin{prop} The flow $\phi(\m{x})$ generated by the the vector field $xy\bl y^2$ is rational and is given by the expression
\begin{eqnarray*}
\phi(\m{x})=\frac{x}{1-y}\bl\frac{y}{1-y}.
\end{eqnarray*}
\end{prop}
Next, the following is an algebraic flow of type $(1,-2)$; see Theorem \ref{thm3}.
\begin{prop}
The flow $\Phi(\m{x})=U(x,y)\bl V(x,y)$ generated by the vector field $(-4x^2+3xy)\bl(-2xy+y^2)$ is algebraic and is given by the expression
\begin{eqnarray*}
\Phi(\m{x})=\frac{y\sqrt{4x+(y-1)^2}+y^2+2x-y}{8x+2(y-1)^2}\bl \frac{y}{\sqrt{4x+(y-1)^2}}.
\end{eqnarray*}
The orbits of this flow are genus $0$ curves $u^{-1}(u-v)^{-1}v^{4}=\mathrm{const.}$
\label{prop-alg}
\end{prop}
Next is an unramified elliptic flow, which was treated in detail in \cite{alkauskas-un}.
\begin{prop}The flow $\Lambda(\m{x})=\lambda(x,y)\bl\lambda(y,x)$ generated by the vector field $(x^2-2xy)\bl(-2xy+y^2)$ is elliptic and can be given the analytic expression 
\begin{eqnarray*}
\Lambda(\m{x})=\frac{\v\big{(}c\v^2-sc^2y\v+s^2xy\big{)}^{2}}
{y\big{(}x-c^3y\big{)}\big{(}c^2\v^2-sx\v+s^2cxy\big{)}}
\bl\frac{\v\big{(}c^2\v^2-sx\v+s^2cxy\big{)}^{2}}
{x\big{(}x-c^3y\big{)}\big{(}c\v^2-sc^2y\v+s^2xy\big{)}};
\end{eqnarray*}
here $\v=\v(x,y)=[xy(x-y)]^{1/3}$, and $s=\sm(\v),c=\cm(\v)$ are the Dixonian elliptic functions \cite{alkauskas-un,flajolet-b}. Thus, these are the functions which parametrize the Fermat cubic $x^{3}+y^{3}=1$, and also $\sm(0)=0$, $\sm'=\cm^{2}$. The orbits of this flow are genus $1$ curves $xy(x-y)=\mathrm{const}$. 
\label{prop-ell}
\end{prop}
Next, we present an analytic expression for a non-elliptic abelian flow of type I whose orbits are genus $2$ curves. The expression for flows with orbits of the form $x^A(x-y)^{B}y^C=\mathrm{const.}$ whose genus is $g\geq 2$ is analogous.
\begin{prop}
The flow $\Delta(\m{x})=\vartheta(x,y)\bl \xi(x,y)$ generated by the vector field $(2x^2-4xy)\bl(-3xy+y^2)$ is non-elliptic abelian flow and is given by the analytic expression
\begin{eqnarray*}
\Delta(\m{x})=
\frac{\k^{4/5}\big{(}\alpha(\frac{x}{y})-\v\big{)}\v}{\Big{(}\k\big{(}\alpha(\frac{x}{y})-\v\big{)}-1\Big{)}^{2/5}}\bl
\frac{\v}{\k^{1/5}\big{(}\alpha(\frac{x}{y})-\v\big{)}\cdot\Big{(}\k\big{(}\alpha(\frac{x}{y})-\v\big{)}-1\Big{)}^{2/5}}.
\end{eqnarray*}
Here $\v=\v(x,y)=[x(x-y)^{2}y^{2}]^{1/5}$, $\alpha$ is an abelian integral
\begin{eqnarray*}
\alpha(x)=\frac{1}{5}
\int\limits_{1}^{\frac{1}{1-x}}\frac{\d t}{t^{3/5}(t-1)^{4/5}},
\end{eqnarray*}
and $\k$ is an abelian function, the inverse of $\alpha$. The orbits of this flow are genus $2$ curves $x(x-y)^{2}y^{2}=\mathrm{const.}$ In the special case, one has
\begin{eqnarray*}
\frac{\vartheta(x,-x)}{\xi(x,-x)}=\k\big{(}c-4^{1/5}x\big{)},\quad c=\alpha(-1).
\end{eqnarray*}
The pair of abelian functions $\big{(}\k(z),\k'(z)\big{)}$ parametrizes (locally) the genus $2$ singular curve $5^{5}(1-x)^{3}x^{4}=y^{5}$. In particular, one can give an alternative expression
\begin{eqnarray*}
\Delta(\m{x})=
\frac{5^{2/3}\k^{4/3}\big{(}\alpha(\frac{x}{y})-\v\big{)}\v}{\k'\,^{2/3}\big{(}\alpha(\frac{x}{y})-\v\big{)}}\bl
\frac{5^{2/3}\k^{1/3}\big{(}\alpha(\frac{x}{y})-\v\big{)}\v}{\k'\,^{2/3}\big{(}\alpha(\frac{x}{y})-\v\big{)}}.
\end{eqnarray*}
\label{prop-abel}
\end{prop}
How we should understand this solution as a flow, at least for real $x,y$? What we need in fact is to uniquely define the values for $z^{-1}\vartheta(xz,yz)$ and $z^{-1}\xi(xz,yz)$ for $z$ small enough, subject to a condition that at $z=0$ these branches are equal to $x$ and $y$, respectivelly. Fix $x,y$, put $(x,y)\mapsto(xz,yz)$, and vary only $z$. Choose one of the possible values of $\alpha(\frac{x}{y})$. Then $\v=zx^{1/5}(x-y)^{2/5}y^{2/5}$, and for $z$ small enough, choose a branch of $\k(\alpha(\frac{x}{y})-zx^{1/5}(x-y)^{2/5}y^{2/5})$ which is equal to $\frac{x}{y}$ at $z=0$.\\

The following is an example of an abelian flow of type II.
\begin{prop}
The flow $q(\m{x})$, generated by the vector field
\begin{eqnarray*} 
\varpi(x,y)\bl\varrho(x,y)=\frac{x^4}{y^2}+xy\bl 0. 
\end{eqnarray*}
can be given the analytic expression
\begin{eqnarray*}
q(\m{x})=\frac{xye^{y}}{(x^3+y^3-x^{3}e^{3y})^{1/3}}\bl y.
\end{eqnarray*}
The orbits of this flow are lines $y=\mathrm{const.}$
\end{prop}
We can also integrate the vector field
\begin{eqnarray*} 
\varpi(x,y)\bl\varrho(x,y)=\frac{x^3}{y}\bl 0, 
\end{eqnarray*}
to obtain an algebraic flow
\begin{eqnarray*}
r(\m{x})=\frac{x}{\sqrt{1-\frac{2x^{2}}{y}}}\bl y.
\end{eqnarray*}
However, the numerator of $\varpi(x,y)$ has a multiple root, and therefore the vector field is not reduced (see Subsection \ref{sub-alg-II}).\\

And finally, we present an integral flow which is non-abelian.
\begin{prop}
The flow $\Xi(\m{x})=G(x,y)\bl H(x,y)$ generated by the vector field $(x^2+xy+y^2)\bl(xy+y^2)$ is non-abelian integral flow and is given by the expression
\begin{eqnarray*}
G(x,y)\bl H(x,y)=\psi\v\exp\Big{(}\psi+\frac{\psi^{2}}{2}\Big{)}\bl \v\exp\Big{(}\psi+\frac{\psi^{2}}{2}\Big{)};\\
\text{ here } 
\psi=\l\Big{(}\beta\Big{(}\frac{x}{y}\Big{)}-\v\Big{)},
\v=\exp\Bigg{(}-\frac{x}{y}-\frac{x^2}{2y^2}\Bigg{)}y,
\end{eqnarray*}
and where $\beta(x)$ is the error function
\begin{eqnarray*}
\beta(x)=-\frac{\sqrt{\pi e}}{\sqrt{2}}\mathrm{erf}\Bigg{(}\frac{x+1}{\sqrt{2}}\Bigg{)},\quad
\mathrm{erf}(x)=\frac{2}{\sqrt{\pi}}\int\limits_{0}^{x}e^{-t^2}\d t,
\end{eqnarray*} 
and $\l$ is the inverse of $\beta$. The orbits of this flow are transcendental curves $\v=\mathrm{const.}$
In the special case,
\begin{eqnarray*}
\frac{G(x,-x)}{x}=-\l(xe^{1/2})\exp\Big{(}\frac{1}{2}\big{(}\l(xe^{1/2})+1\big{)}^2\Big{)}.
\end{eqnarray*}
\label{prop-int}
\end{prop}

Non-integral and non-abelian vector fields still produce valid locally well-defined flows. Indeed, the series (\ref{series}) makes sense and is locally uniquely defined for any pair of $2-$homogenic rational functions $\varpi\bl\varrho$. For example, the vector field $\sqrt{2}xy\bl(-y^{2})$ produces the flow
\begin{eqnarray*}
x(y+1)^{\sqrt{2}}\bl\frac{y}{y+1}.
\end{eqnarray*} 
Yet, the arithmetic of these flows is not in any sense finite, so these are not so interesting to investigate from a number-theoretic or algebro-geometric point of view, and they have no deeper structure.\\

In \cite{alkauskas-un} we posed 8 problems, all related to projective flows and the projective translation equation. Here we add two more to this list.

\begin{prob9}Let $n\geq 3$. Describe all $n$ dimensional projective flows over $\mathbb{C}$ with rational vector fields such that all their coordinates are given by algebraic functions.
\end{prob9}

\begin{prob10}Let $n\geq 3$. Describe all  $n$ dimensional projective flows over $\mathbb{C}$ with rational vector fields such that their orbits are algebraic curves.
\end{prob10}

This list is still continued in \cite{alkauskas-comm} with the \textbf{Problem 11}. The latter concerns projective flows $\phi$ and $\psi$ with rational vector fields such that $\phi\circ\psi=\psi\circ\phi$, which is equivalent to the property that the composition is a flow again. It appears that in a $2-$dimensional case this implies that both $\phi$ and $\psi$ are algebraic flows with very special properties.

\section{General setting}

In the most general case, if we do not encounter an obstruction \cite{alkauskas}, the vector field, with the help of conjugation with a $1-$BIR, can be transformed into a pair of quadratic forms $\varpi(x,y)\bl\varrho(x,y)$. The case when $y\varpi-x\varrho$ has a double or triple root is dealt in Section \ref{sect-int}. If there is no triple root, we can further reduce the vector field to \cite{alkauskas}
\begin{eqnarray*}
\varpi(x,y)\bl\varrho(x,y)=x^{2}+Bxy\bl Cxy+y^2.
\end{eqnarray*}
Assume that $B,C\neq 1$, otherwise $y\varpi-x\varrho$ has a double root. Consider conjugation with the linear change 
$(x,y)\mapsto(\frac{B-1}{C-1}x,y)$. Conjugating vector field with this, we obtain
\begin{eqnarray}
\varpi'(x,y)\bl\varrho'(x,y)=\frac{B-1}{C-1}x^{2}+Bxy\bl \frac{B-1}{C-1}Cxy+y^2.
\label{ovaa}
\end{eqnarray} 
Let $\phi(\m{x})$ be a flow with the vector field $\varpi'(x,y)\bl\varrho'(x,y)$. Note that the sum of the coefficients 
for both components is constant:
\begin{eqnarray*}
\frac{B-1}{C-1}+B=\frac{(B-1)C}{C-1}+1=\frac{BC-1}{C-1}.
\end{eqnarray*}
Thus, despite a huge arithmetic variety and complexity of such flows, on three exceptional lines $x=0$, $y=0$ and $x=y$ the flow can be given a simple closed-form expression.
\begin{prop}For such flow $\phi$, we have:
\begin{eqnarray*}
\phi(x,0)=\frac{x}{1-\frac{B-1}{C-1}x}\bl 0,\quad
\phi(0,x)=0\bl\frac{x}{1-x},\quad 
\phi(x,x)=\frac{x}{1-\frac{BC-1}{C-1}x}\bl
\frac{x}{1-\frac{BC-1}{C-1}x}.
\end{eqnarray*}
\label{prop-sing}
\end{prop}
Next,
\begin{eqnarray*}
\varpi'(x,1)-x\varrho'(x,1)=(1-B)x(x-1).
\end{eqnarray*}
And
\begin{eqnarray*}
-\frac{\varrho'(x,1)}{\varpi'(x,1)-x\varrho'(x,1)}=
\frac{1}{(1-B)x}+\frac{BC-1}{(B-1)(C-1)(x-1)}.
\end{eqnarray*}
So, the orbits of the flow with the vector field $\varpi'(x,y)\bl\varrho'(x,y)$ are given by (in (\ref{orbits}) we take $N=1$)
\begin{eqnarray}
x^{1-C}(x-y)^{BC-1}y^{1-B}\equiv\mathscr{W}^{(1-C)(1-B)}(x,y)=\mathrm{const}.
\label{orbitos}
\end{eqnarray} 
So, we obtain the following. The flows with the vector fields $(x^2+\mathbf{b}xy,\mathbf{c}xy+y^2)$ are linearly conjugate if $(\mathbf{b},\mathbf{c})$ is any pair from an unordered triple \cite{alkauskas}
\begin{eqnarray*}
B,C,\frac{B+C-2}{BC-1}.
\end{eqnarray*}
We may put this into a symmetric form. Denote all the three different numbers $B$, $C$ and 
$\frac{B+C-2}{BC-1}$ in the above display by $\mathbf{a}$, $\mathbf{b}$ and $\mathbf{c}$, in any order. Then
\begin{eqnarray}
\frac{1}{1-\mathbf{a}}+\frac{1}{1-\mathbf{b}}+\frac{1}{1-\mathbf{c}}=1.
\label{cyk}
\end{eqnarray}
And the level of such flow is $N=\alpha+\beta+\gamma$, where
\begin{eqnarray*}
(1-C:BC-1:1-B)=(\alpha:\beta:\gamma),\quad\alpha,\beta,\gamma\in\mathbb{Z},\quad\mathrm{g.c.d}(\alpha,\beta,\gamma)=1.
\end{eqnarray*}
\section{Algebraic flows}
\subsection{Example 1. }The canonical solution $\phi_{N}(x,y)$ (see Theorem \ref{thm1}) is an algebraic flow if $N=\frac{p}{q}$, $p,q\in\mathbb{Z}$ $(p,q)=1$. The orbits of this flow are given by genus $0$ curves $x^{q}y^{p}\equiv\textrm{const.}$ But we will soon see that these examples do not exhaust all algebraic flows!
 
\subsection{Example 2. }We will now construct algebraic flow which is non-conjugate to any canonical solution. Indeed, now we explore the case $(B,C)=(3,\frac{1}{2})$; So, looking at (\ref{ovaa}) we see that  we need to explore the vector field
\begin{eqnarray*}
(-4x^2+3xy)\bl(-2xy+y^2). 
\end{eqnarray*}
This is an example of the flow with second degree of ramification, whose orbits are genus $0$ quartic curves $u^{-1}(u-v)^{-1}v^{4}=\mathrm{const.}$, as is implied by (\ref{orbitos}).

\begin{figure}
\epsfig{file=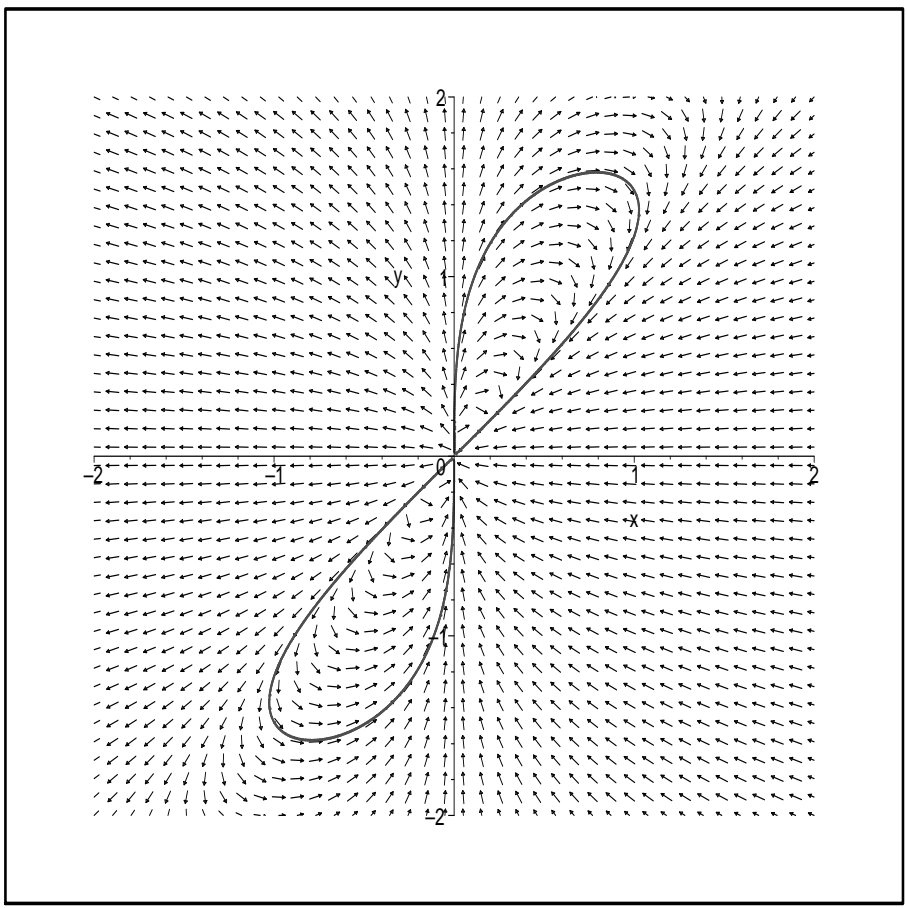,width=340pt,height=340pt,angle=-90}
\caption{The normalized vector field of the algebraic flow
$\Phi(\m{x})$, $(x,y)\in[-2,2]^{2}$. The orbit shown is 
$x(x-y)=-0.1\cdot y^{4}$.}
\label{figure2}
\end{figure}

Let us apply (\cite{alkauskas}, Proposition 2). We have (for typographic reasons, now we discard the prime sign for both $\varpi$ and $\varrho$):
\begin{eqnarray*}
\frac{y\varpi_{y}(x,y)-x\varrho_{y}(x,y)}{y\varpi_{x}(x,y)-x\varrho_{x}(x,y)}=
-\frac{x(y+2x)}{3y(2x-y)},
\end{eqnarray*}
and since the latter is not a M\"{o}bius transformation, $\Phi(\m{x})=U(x,y)\bl V(x,y)$ is not a rational flow.\\

But we have seen in (\cite{alkauskas}, Section 2.4) that the vector field of the rational flow
\begin{eqnarray}
\phi^{(2)}_{3}(x,y)=\frac{(y^2+x)^3}{(x+2xy+y^3)^2}\bl\frac{y(y^2+x)}{(x+2xy+y^3)}=u(x,y)\bl v(x,y)
\label{phi2-3}
\end{eqnarray}
is given by
\begin{eqnarray*}
\tilde{\varpi}(x,y)\bl\tilde{\varrho}(x,y)=(-4xy+3y^2)\bl(-2xy^2+y^3)x^{-1},
\end{eqnarray*}
and since
\begin{eqnarray*}
\frac{\varpi(x,y)}{\varrho(x,y)}=\frac{\tilde{\varpi}(x,y)}{\tilde{\varrho}(x,y)},
\end{eqnarray*}
the flows $\Phi$ and $\phi^{(2)}_{3}$ share the same orbits!  For example, as we know,
\begin{eqnarray*}
\frac{V(xz,yz)^4}{z^{2}\cdot U(xz,yz)\cdot\big{(}U(xz,yz)-V(xz,yz)\big{)}}\equiv \frac{y^4}{x(x-y)}.
\end{eqnarray*}
Nevertheless, since the flow $\Phi$ is linearly conjugate to the flow with orbits $x^2+Bxy\bl Cxy+y^2$, we know in advance that the flow $\Phi(\m{x})$ has a quadratic ramification \cite{alkauskas-un}. Now we will find the flow $\Phi$ explicitly.\\
 
Let, as before, $\varpi(x,y)\bl\varrho(x,y)=(-4x^2+3xy)\bl(-2xy+y^2)$. Solving the differential equation (\ref{33}) we obtain that the $1$-BIR $xA(x,y)\bl yA(x,y)$ transforms the flow $\Phi$ into an uni-variate flow \cite{alkauskas}, where $A(x,y)=\frac{2x}{y}-1$. So (\cite{alkauskas}, Proposition 3)
\begin{eqnarray*}
A(x,y)\varpi(x,y)-A_{y}\Big{(}x\varrho(x,y)-y\varpi(x,y)\Big{)}&=&\frac{-4x^3+6x^{2}y-3xy^{2}}{y},\\
A(x,y)\varrho(x,y)+A_{x}\Big{(}x\varrho(x,y)-y\varpi(x,y)\Big{)}&=&-y^2.
\end{eqnarray*}
So, if we put 
\begin{eqnarray*}
\ell(x,y)=xA(x,y)\bl yA(x,y)=\frac{2x^2-xy}{y}\bl 2x-y,\quad \ell^{-1}(x,y)=\frac{xy}{2x-y}\bl \frac{y^2}{2x-y},
\end{eqnarray*}
then the flow $\ell^{-1}\circ\Phi\circ\ell(x,y)$
has the vector field
\begin{eqnarray*}
\frac{-4x^3+6x^{2}y-3xy^{2}}{y}\bl(-y^2)=\eta(x,y)\bl \nu(x,y).
\end{eqnarray*}
The orbits of this vector field are given by
\begin{eqnarray*}
\mathscr{W}_{\mathbf{a}}(x,y):=\frac{(4x^2-4xy+y^2)y^{2}}{x(x-y)}=\mathrm{const.}
\end{eqnarray*}
As the method of integration, let us consider the following system of differential equations:
\begin{eqnarray}
\left\{\begin{array}{l}
p'(t)=-\eta(p(t),q(t)),\\
q'(t)=-\nu(p(t),q(t)),\\
\mathscr{W}_{\mathbf{a}}(p(t),q(t))=1.
\end{array}
\right.
\label{sys}
\end{eqnarray}
In particular, since $\nu(x,y)=-y^{2}$, we put $q(t)=-t^{-1}$. So, the last identity gives
\begin{eqnarray*}
p^2(4t^2-t^4)+p(4t-t^3)+1=0.
\end{eqnarray*}
Solving, we get
\begin{eqnarray*}
p(t)=-\frac{1}{2t}-\frac{1}{2\sqrt{t^{2}-4}}.
\end{eqnarray*}
We check directly that the first identity of (\ref{sys}) is satisfied (which is automatic in a general: the second and the third yield the first). So, applying the general method to solve the PDE (\ref{g-parteq}), developed in \cite{alkauskas-un}, we have
\begin{eqnarray*}
u\Big{(}p(a)\v,q(a)\v\Big{)}=p(a-\v)\v,&\quad& v\Big{(}p(a)\v,q(a)\v\Big{)}=q(a-\v)\v,\\ 
x=p(a)\v,&\quad& y=q(a)\v.
\end{eqnarray*}
These equations imply, for example, that
\begin{eqnarray*}
W_{\mathbf{a}}(x,y)=W_{\mathbf{a}}\Big{(}p(a)\v,q(a)\v\Big{)}=\v^{2}.
\end{eqnarray*}
So, the identity for $v$ is what is expected:
\begin{eqnarray*}
v\Big{(}p(a)\v,-\frac{\v}{a}\Big{)}=\frac{\v}{\v-a}\Longrightarrow v(x,y)=\frac{y}{y+1}.
\end{eqnarray*}
The identity for $u$ is more complicated:
\begin{eqnarray*}
u\Big{(}-\frac{\v}{2a}-\frac{\v}{2\sqrt{a^2-4}},-\frac{\v}{a}\Big{)}=\frac{\v}{2(\v-a)}-\frac{\v}{2\sqrt{(a-\v)^2-4}}.
\end{eqnarray*}
If we replace $\v=-ay$, we obtain:
\begin{eqnarray}
u\Big{(}\frac{y}{2}+\frac{ay}{2\sqrt{a^2-4}},y\Big{)}=\frac{y}{2(y+1)}+\frac{ay}{2\sqrt{(a+ay)^2-4}}.
\label{L}
\end{eqnarray}
Now, 
\begin{eqnarray*}
\frac{y}{2}+\frac{ay}{2\sqrt{a^2-4}}=x\Longrightarrow \frac{a}{2\sqrt{a^{2}-4}}=\frac{x}{y}-\frac{1}{2}
\Longrightarrow a=\frac{2x-y}{\sqrt{x^2-xy}}.
\end{eqnarray*}
So, finally substituting this into (\ref{L}), we obtain

\begin{eqnarray*}
u(x,y)=\frac{y}{2y+2}+\frac{2xy-y^2}{2\sqrt{4xy(x-y)(y+2)+y^2(y+1)^2}}.
\end{eqnarray*}

Returning to the vector field $(-4x^2+3xy)\bl(-2xy+y^2)$, we see that its  flow is given by $\ell\circ(u\bl v)\circ\ell^{-1}(x,y)$. Calculating with MAPLE again, we obtain Proposition \ref{prop-alg}. The closed-form expression for $\Phi$ has no immediate resemblance to $\phi^{(2)}_{3}$ as is given by (\ref{phi2-3}), so it is indeed surprising that these two share the same orbits! The double-check with MAPLE shows that
\begin{eqnarray*}
U_{x}(x,y)\cdot(-4x^2+3xy-x)+U_{y}(x,y)\cdot(-2xy+y^2-y)+U(x,y)\equiv 0,
\end{eqnarray*}
identicaly so for $V$, and the boundary conditions are satisfied. This is a completely new algebraic flow, not conjugate to any canonical flow $\phi_{N}$, $N\in\mathbb{Q}$. We can also double-check the validity of Theorem \ref{thm3} and thus triple-check this example in particular.\\

Indeed, in this case 
\begin{eqnarray*}
P_{1,-2}(x)=1-2x.
\end{eqnarray*}
So, first we check that
\begin{eqnarray}
\frac{1}{V}-\frac{2U}{V^2}=\frac{1}{y}-\frac{2x}{y^2}-1,
\label{nQ}
\end{eqnarray}
which does indeed hold, and also, with MAPLE, the orbit conservation property
\begin{eqnarray*}
U^{-1}(U-V)^{-1}V^4=x^{-1}(x-y)^{-1}y^4.
\end{eqnarray*}
On the other side, the latter property together with (\ref{nQ}), if solved explicitly, gives exactly the values of $U$ and $V$, given in Proposition \ref{prop-alg}. 

\subsection{Example 3. }We note that yet another algebraic flow appeared in the very first paper on projective translation equation \cite{alkauskas-t}. Namely, consider the canonical flow $\phi(\m{x})=\frac{x}{x+1}\bl\frac{y}{y+1}$. Now, consider the 1-homogenic transformation
\begin{eqnarray*}
\ell(x,y)=\frac{x^3}{x^2+y^2}\bl\frac{y^3}{x^2+y^2}.
\end{eqnarray*}
This maps the unit circle to the astroid $|x|^{2/3}+|y|^{2/3}=1$. The inverse of this map is not rational and it is given by
\begin{eqnarray*}
\ell^{-1}(x,y)=x+(xy^2)^{1/3}\bl y+(x^2y)^{1/3}.
\end{eqnarray*}
We now calculate that the flow $\ell^{-1}\circ\phi\circ\ell(\m{x})$ is given by
$\chi(x,y)\bl \chi(y,x)$, where
\begin{eqnarray*}
\chi(x,y)=\frac{x^3}{x^3+x^2+y^2}+\frac{xy^2}{(x^3+x^2+y^2)^{1/3}(y^3+x^2+y^2)^{2/3}}.
\end{eqnarray*} 
Of course, $\ell$ is not $1-$BIR. However, the vector field of this flow is given by
\begin{eqnarray*}
-\frac{x(3x^5+x^3y^2+2y^5)}{3(x^2+y^2)^2}\bl-\frac{y(3y^5+y^3x^2+2x^5)}{3(x^2+y^2)^2};
\end{eqnarray*}
so, it is rational again! The orbits of this flow are given by genus $0$ curves
\begin{eqnarray*}
\mathscr{W}(u,v)=\frac{u^3v^{3}}{(u-v)(u^2+v^2)(u^2+uv+v^2)}=\textrm{const.}
\end{eqnarray*}


\section{Elliptic unramified and ramified flows}
\label{ell-un}

In \cite{alkauskas-un} we explored the case $(B,C)=(-2,-2)$. So, this gives the vector field
\begin{eqnarray*}
(x^2-2xy)\bl(-2xy+y^2),
\end{eqnarray*}
which is exceptional among all vector fields given by a pair of quadratic forms since this vector field is invariant under two linear involutions, thus producing a $S_{3}$ group as the group of its symmetries. This is an example of the flow without ramifications, and whose orbits are genus $1$ cubic curves $u(u-v)v\equiv\mathrm{const.}$ Since, as said, this was treated in detail in \cite{alkauskas-un}, we only have reproduced the main result as Proposition \ref{prop-ell}.\\

\begin{figure}
\epsfig{file=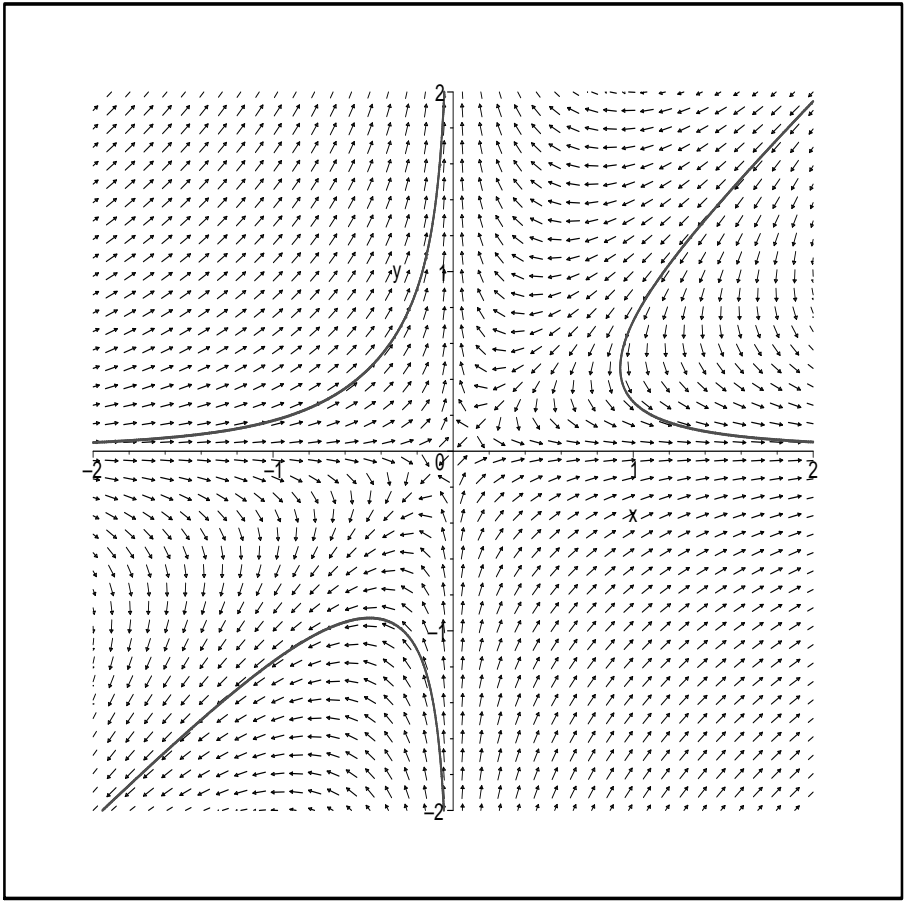,width=340pt,height=340pt,angle=-90}
\caption{The normalized vector field of the elliptic unramified flow $\Lambda(\m{x})$, $(x,y)\in[-1,1]^{2}$. The orbit shown is $xy(x-y)=0.2$.}
\label{figure3}
\end{figure}

Next, we can similarly explore the case $(B,C)=(5,-\frac{1}{3})$ and the vector field (\ref{ovaa}) 
\begin{eqnarray*}
\varpi(x,y)\bl\varrho(x,y)=(-3x^2+5xy)\bl(xy+y^2).
\end{eqnarray*}
This is an example of the flow with ramification whose orbits are genus $1$ quintic curves  $u^{-1}(u-v)^{2}v^{3}\equiv\mathrm{const.}$; see (\ref{orbitos}).

\begin{figure}
\epsfig{file=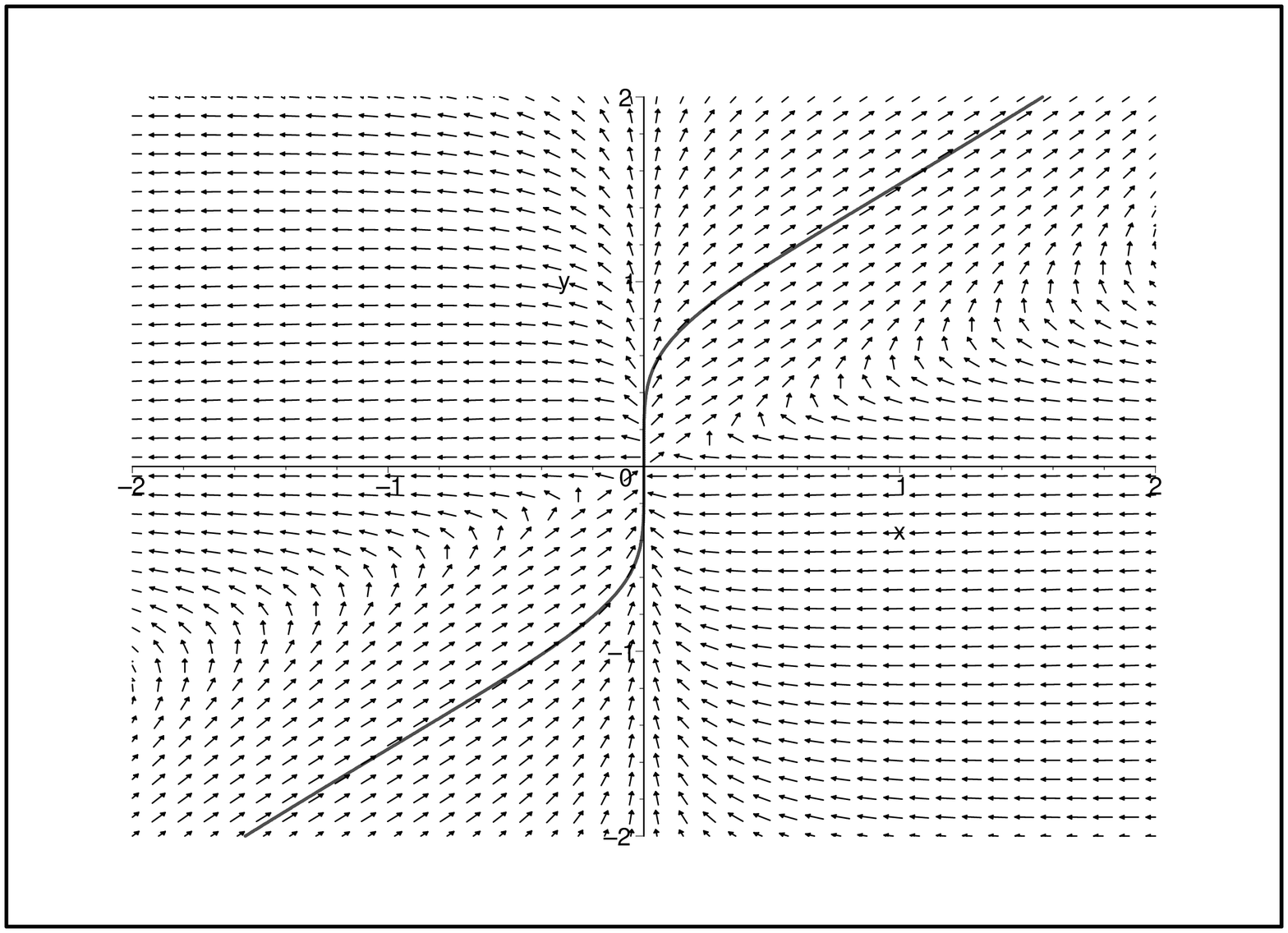,width=340pt,height=340pt,angle=-90}
\caption{The normalized vector field of the elliptic ramified flow $\Gamma(\m{x})$, $(x,y)\in[-2,2]^{2}$. The orbit shown is 
$(x-y)^2y^3=x$.}
\label{figure4}
\end{figure}

\begin{prop}
The flow $\Gamma(\m{x})$ generated by the vector field $(-3x^2+5xy)\bl(xy+y^2)$ is elliptic. The orbits of this flow are genus $1$ curves $x^{-1}(x-y)^{2}y^{3}=\mathrm{const.}$
\end{prop} 
We do not provide analytic formulas for the solution, since these are completely analogous to the ones given in Proposition \ref{prop-ell}.

\section{Abelian non-elliptic flow}
\label{abel-non-ell}
Here we explore the case $(B,C)=(-4,-\frac{3}{2})$ and so the vector field
\begin{eqnarray*}
\varpi(x,y)\bl\varrho(x,y)=(2x^2-4xy)\bl(-3xy+y^2).
\end{eqnarray*} 
This implies $x\varrho-y\varpi=-5x^2y+5xy^2$. This is an example of the flow with a ramification, whose orbits are genus $2$ quintic curves $u(u-v)^2v^2\equiv\mathrm{const.}$ The vector field and the selected orbit is shown in Figure \ref{figure5}.\\

\begin{figure}
\epsfig{file=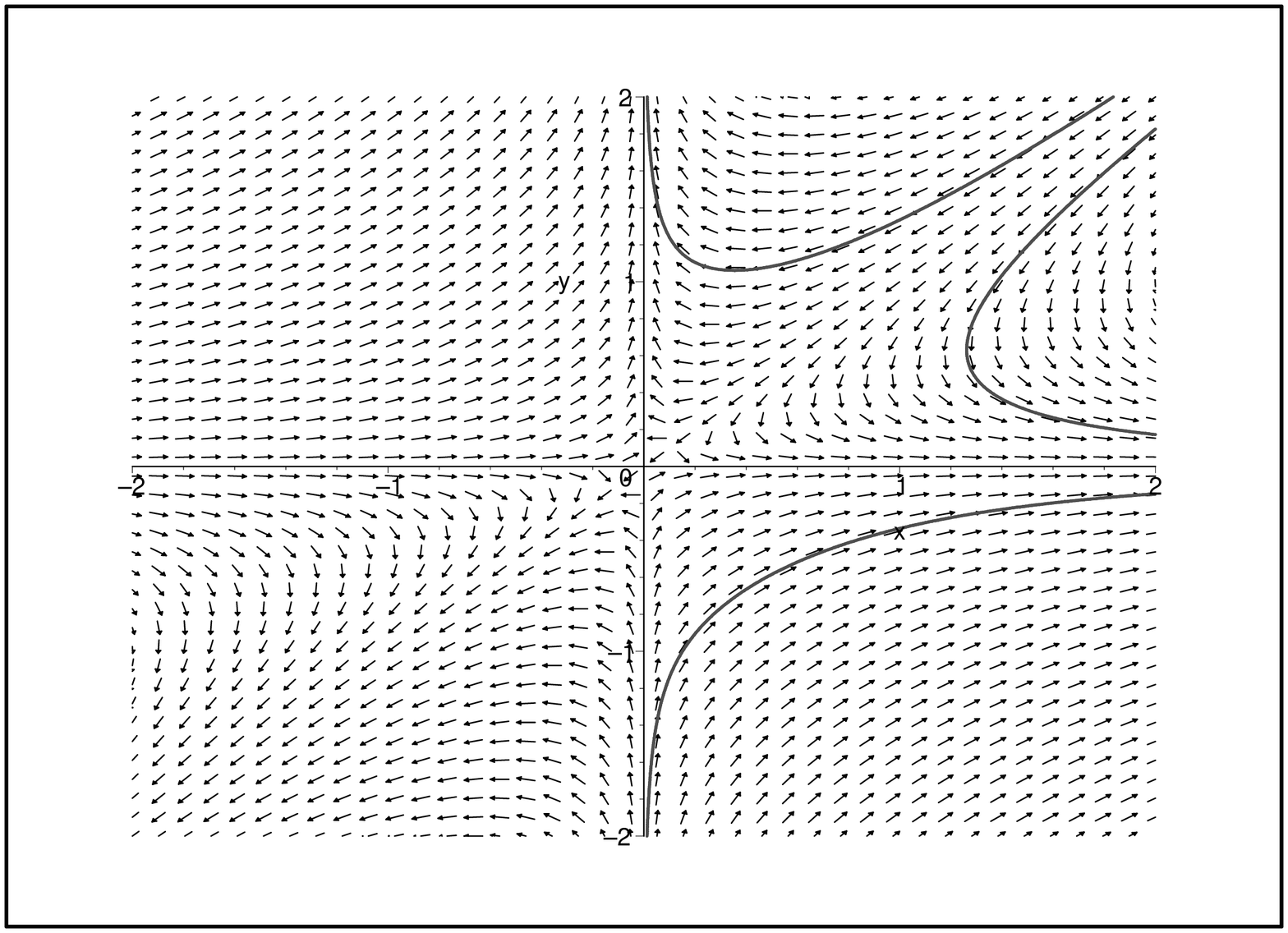,width=340pt,height=340pt,angle=-90}
\caption{The normalized vector field of the abelian non-elliptic flow $\Delta(\m{x})$, $(x,y)\in[-2,2]^{2}$. The orbit shown is 
$x(x-y)^{2}y^{2}=0.2$.}
\label{figure5}
\end{figure}

We know that (\cite{alkauskas}, Section 3.2) if we define recurrently
\begin{eqnarray}
\varpi^{(i+1)}(x,y)=\frac{1}{i}\,[\varpi_{x}^{(i)}(x,y)\varpi(x,y)+\varpi_{y}^{(i)}\varrho(x,y)],\quad i\geq 2,
\label{reku}
\end{eqnarray}
analogously for $\varrho^{(i)}(x,y)$, then
\begin{eqnarray}
\vartheta(xz,yz)=xz+\sum\limits_{i=2}^{\infty}z^{i}\varpi^{(i)}(x,y),\quad
\xi(xz,yz)=yz+\sum\limits_{i=2}^{\infty}z^{i}\varrho^{(i)}(x,y).
\label{series}
\end{eqnarray}
In particular, experimenting numerically, we get that on the line $x=y$ we have a particularly simple expression:
\begin{eqnarray*}
\vartheta(x,x)=\sum\limits_{i=1}(-1)^{i-1}2^{i-1}x^{i}=\frac{x}{1+2x};\quad \xi(x,x)=\frac{x}{1+2x}.
\end{eqnarray*}
Moreover,
\begin{eqnarray*}
\Delta(x,0)=\frac{x}{1-2x}\bl 0,\quad \Delta(0,x)=0 \bl \frac{x}{1-x}.
\end{eqnarray*}
Of course, we know this rigorously, and it follows from Proposition (\ref{prop-sing}).
On the other hand, using (\ref{reku}), we can easily calculate
\begin{eqnarray}
\frac{\vartheta(x,-x)}{x}=1+6x+16x^{2}+36x^{3}+111x^{4}+369x^{5}+\frac{2243}{2}x^{6}+\frac{46101}{14}x^{7}+\cdots,\label{theta-exp}\\
\frac{\xi(x,-x)}{x}=-1+4x-x^{2}-x^{3}-\frac{69}{4}x^{4}+\frac{37}{4}x^{5}+
\frac{79}{8}x^{6}+\frac{793}{56}x^{7}-\frac{48931}{448}x^{8}-\cdots.\nonumber
\end{eqnarray}  
In particular,
\begin{eqnarray}
\frac{\vartheta(x,-x)}{\xi(x,-x)}=
-1-10x-55x^{2}-245x^{3}-\frac{4035}{4}x^{4}
-\frac{15763}{4}x^{5}-\frac{118275}{8}x^6-
\frac{3017075}{56}x^7\nonumber\\
-\frac{86027325}{448}x^8-\frac{43032775}{64}x^9-\frac{2079392255}{896}x^{10}-\frac{78011676535}{9856}x^{11}-\cdots.
\label{expann}
\end{eqnarray} 
Soon we will clear out what kind of function is this last expansion (\ref{expann}) (in fact, an abelian function).                           
\subsection{Special hypergeometric function} As we noticed in \cite{alkauskas-un}, the hypergeometric function $\,_{2}F_{1}(\frac{2}{3},1;\frac{4}{3};x)$ plays an important role in the description of the flow $\Lambda(\m{x})$, defined in the Subsection \ref{ell-un}.\\

Equally, in the description of the flow $\Delta(\m{x})$ the following function  plays a crucial role:
\begin{eqnarray}
Y(x)=\frac{1}{5}\int\limits_{0}^{1}\frac{\d t}{(1-t)^{4/5}(1-xt)^{3/5}}
=\,_{2}F_{1}\Big{(}\frac{3}{5},1;\frac{6}{5};x\Big{)},\quad -\infty<x<1.
\label{int}
\end{eqnarray}
It satisfies the initial condition $Y(0)=1$, and the linear ODE
\begin{eqnarray*}
5x(1-x)Y'(x)+(1-3x)Y(x)=1.
\end{eqnarray*}
Again, as in \cite{alkauskas-un}, we notice that the derivative of this gives
\begin{eqnarray*} 
5x(1-x)Y''(x)+(6-13x)Y'(x)-3Y(x)=0,
\end{eqnarray*}
which coincides with Euler's hypergeometric differential equation for $(a,b;c)=(\frac{3}{5},1;\frac{6}{5})$. Note that this always happens, that is, $\,_{2}F_{1}(a,b;c,x)$ satisfies non-homogenic first order ODE, in case $b=1$. The change $1-xt=(1-x)T$ in (\ref{int}) gives
\begin{eqnarray}
Y(x)=\frac{1}{5x^{1/5}(1-x)^{2/5}}
\int\limits_{1}^{\frac{1}{1-x}}\frac{\d t}{t^{3/5}(t-1)^{4/5}},\quad 0<x<1.
\label{invv}
\end{eqnarray}

We can see the role of $Y(x)$ in the analytic description of the flow $\Delta(\m{x})$ from the following
\begin{prop}Let us define the curve $\mathscr{K}$ parametrically by 
\begin{eqnarray*}
\mathscr{K}=\Big{\{}\Big{(}xY(x),Y(x)\Big{)}:-\infty<x<1\Big{\}}.
\end{eqnarray*}
Then the function $\vartheta(x,y)$ vanishes on the curve $\mathscr{K}$, and $\xi(x,y)$ on this curve takes the value $\infty$:
\begin{eqnarray*}
\vartheta(xY(x),Y(x))\equiv 0,\quad \xi(xY(x),Y(x))=\infty\text{ for }-\infty<x<1.
\end{eqnarray*}
\label{zero-loc}
\end{prop}
\begin{proof}We will not present the full proof, since it is completely analogous to (\cite{alkauskas-un}, Proposition 5). Indeed, if we put
\begin{eqnarray*}
E(x)=\vartheta(xY(x),Y(x)),\quad \a=xY(x),\quad \b=Y(x),
\end{eqnarray*}
then $E(0)=\vartheta(0,1)=0$ (Proposition \ref{prop-sing}), and from the differential equation for $Y(x)$ we infer that
\begin{eqnarray*}
2\a^{2}-4\a\b-\a&=&[Y(x)+xY'(x)]\cdot Y(x)5x(x-1),\\
-3\a\b+\b^2-\b&=&Y'(x)\cdot Y(x)5x(x-1).
\end{eqnarray*}
So, this and the differential equation for $\vartheta(x,y)$, namely (\ref{g-parteq}), imply
\begin{eqnarray*}
E'(x)\cdot\Big{(}Y(x)5x(x-1)\Big{)}=-E(x)\text{ for }-\infty<x<1.
\end{eqnarray*}
this gives $E(x)\equiv 0$.
\end{proof}
\begin{Note}We warn in advance that (as it turned out later) the proof of the last Proposition, as well as the proof of (Proposition 4, \cite{alkauskas-un}), both contain one hidden assumption; without it the proof is not entirely valid.  Indeed, using exactly the same method of the proof, but now applied to the flow $U\bl V$ given in Proposition \ref{prop-alg},  we would ``obtain" that 
\begin{eqnarray}
U(x-2x^2,1-2x)\mathop{=}^{?}0,
\label{u-loc}
\end{eqnarray} 
which is not true, as can be checked! So, one can ask: what is wrong? Since analogous Proposition \ref{zero-loc}, as can be checked \emph{a posteriori} using the explicit formulas in Proposition \ref{prop-abel}, and also analogous (Proposition 4, \cite{alkauskas-un}), are both correct. Indeed, plugging $(X,Y)=(xY(x),Y(x))$ into the formula for $\vartheta$ of Proposition \ref{prop-abel}, we obtain:
\begin{eqnarray*}
\mathbf{k}\Big{(}\alpha\big{(}\frac{X}{Y}\big{)}-\v\Big{)}=
\mathbf{k}\Big{(}\alpha(x)-Y(x)x^{1/5}(x-1)^{2/5}\Big{)}=\mathbf{k}(0)=0.
\end{eqnarray*}
Anyway, in this particular example of $U\bullet V$ in Proposition \ref{prop-alg}, one can check that all the details of the proof of (\ref{u-loc}) are correct, but the conclusion is definitelly wrong, and (\ref{u-loc}) does not hold.\\

The answer is simple:
\begin{eqnarray*}
U(x-2x^2,1-2x)=\frac{1}{4\sqrt{x-x^2}}. 
\end{eqnarray*}

So
\begin{eqnarray*}
``U(0,1)"=\lim\limits_{x\rightarrow 0}U(x-2x^2,1-2x)=\infty\neq 
\lim\limits_{y\rightarrow 1}U(0,y)=0.
\end{eqnarray*}
(Of course, from the flow property, we must follow the convention that $\sqrt{(y-1)^2}=1-y$ for $y$ sufficiently small). So, the differential equation for $E(x)=U(x-2x^2,1-2x)$ is correct, only the initial condition $E(0)=0$ is not satisfied. The reason why this does not happen in Proposition 9 is because in the latter case the orbits of the flow are given by $x(x-y)^2y^2=\mathrm{const.}$, orbits in (Proposition 4, \cite{alkauskas-un}) are given by $x(x-y)y=\mathrm{const.}$ So, no denominators occur, while for the flow $U\bullet V$ the orbits are $x^{-1}(x-y)^{-1}y^4=\mathrm{const.}$ Thus, there are denominators, and it gives the discontinuity at $(x,y)=(0,1)$. This a a very subtle assumption, and that is why Proposition \ref{zero-loc} is still correct, but the method itself must be applied with a caution.
\end{Note}
To find the flow $\vartheta\bl\xi$ explicitly, we will uses techniques developed in \cite{alkauskas}. Consider the $1-$homogenic plane transformation 
\begin{eqnarray*}
\ell(x,y)=xY\Big{(}\frac{x}{y}\Big{)}\bl yY\Big{(}\frac{x}{y}\Big{)},\quad
\ell^{-1}(x,y)=xY^{-1}\Big{(}\frac{x}{y}\Big{)}\bl yY^{-1}\Big{(}\frac{x}{y}\Big{)}.
\end{eqnarray*}
(Here, of course, $Y^{-1}=\frac{1}{Y}$, while $\ell^{-1}$ stands for the inverse transformation). Put $A(x,y)=Y(\frac{x}{y})$. Consider the projective flow
\begin{eqnarray*}
\widetilde{\Delta}(\m{x})=P(x,y)\bl Q(x,y)=\ell^{-1}\circ\Delta\circ\ell(\m{x}).
\end{eqnarray*}
Denote its vector field by $\eta(x,y)\bl\nu(x,y)$. We can calculate that (\cite{alkauskas}, Eq. (32))
\begin{eqnarray*}
\eta(x,y)&=&A(x,y)\varpi(x,y)-A_{y}(x,y)\Big{(}x\varrho(x,y)-y\varpi(x,y)\Big{)}\\
&=&Y\Big{(}\frac{x}{y}\Big{)}(2x^2-4xy)
+Y'\Big{(}\frac{x}{y}\Big{)}\frac{x}{y^2}(-5x^2y+5xy^{2})=xy+(5x^2-5xy)Y\Big{(}\frac{x}{y}\Big{)},\\
\nu(x,y)&=&A(x,y)\varrho(x,y)+A_{x}(x,y)\Big{(}x\varrho(x,y)-y\varpi(x,y)\Big{)}\\
&=&Y\Big{(}\frac{x}{y}\Big{)}(-3xy+y^2)+Y'\Big{(}\frac{x}{y}\Big{)}(-5x^2+5xy)=y^2.
\end{eqnarray*}
The reason for introduction of a $1-$homogenic plane transformation $\ell$ is that now the vector field $\nu(x,y)$ depends only on $y$, and so it is can be easily integrated. We now will find the orbits of the projective flow $\widetilde{\Delta}(\m{x})$. First, put 
\begin{eqnarray*}
a=xzY\Big{(}\frac{x}{y}\Big{)},\quad b=yzY\Big{(}\frac{x}{y}\Big{)},\quad
\vartheta=\vartheta(a,b),\quad \xi=\xi(a,b).
\end{eqnarray*}
Then 
\begin{eqnarray*}
P\bl Q:=\frac{P(xz,yz)}{z}\bl\frac{Q(xz,yz)}{z}=\frac{1}{z}\cdot\ell^{-1}\circ\Delta\circ\ell(\m{x}z)=
\frac{\vartheta}{zY(\frac{\vartheta}{\xi})}\bl\frac{\xi}{zY(\frac{\vartheta}{\xi})}.
\end{eqnarray*} 
We know the orbits of the flow $\Delta(\m{x})$; so,
\begin{eqnarray*}
\vartheta(\vartheta-\xi)^{2}\xi^{2}=a(a-b)^{2}b^{2}=x(x-y)^{2}y^{2}Y^{5}\Big{(}\frac{x}{y}\Big{)}z^{5}.
\end{eqnarray*}
(The first equality follows from the orbit conservation property for the flow $\Delta$). This gives
\begin{eqnarray*}
P(P-Q)^2Q^2=\frac{\vartheta(\vartheta-\xi)^2\xi^2}{z^5Y^{5}(\frac{\vartheta}{\xi})}=
\frac{x(x-y)^{2}y^2Y^{5}(\frac{x}{y})z^{5}}{z^5Y^{5}(\frac{\vartheta}{\xi})}=
\frac{x(x-y)^{2}y^2Y^{5}(\frac{x}{y})}{Y^{5}(\frac{P}{Q})}.
\end{eqnarray*}
Note that these calculations are carried analogously  in a general setting, not just in case of the flow $\Delta$, and this is perfectly compatible with (\cite{alkauskas}, Corollary 1). Thus, the orbits of the flow $\widetilde{\Delta}(\m{x})$ are given by $\mathscr{W}_{\mathbf{a}}(u,v)=\textrm{const.}$, where
\begin{eqnarray*}
\mathscr{W}_{\mathrm{a}}(x,y)=x(x-y)^{2}y^{2}Y^{5}\Big{(}\frac{x}{y}\Big{)}.
\end{eqnarray*}
Now, we will solve the differential system (\ref{sys}). Since $\nu(x,y)=y^{2}$, this gives $q'(t)=-q^{2}(t)$, and we chose $q(t)=t^{-1}$. So, in this case the last identity of (\ref{sys}) yields
\begin{eqnarray*}
\big{(}p(t)t\big{)}\cdot\big{(}p(t)t-1\big{)}^{2}\cdot Y^{5}\big{(}p(t)t\big{)}=t^{5}.
\end{eqnarray*}
If we define $\k(t)=p(t)t$, this gives the implicit equation for $k(t)$:
\begin{eqnarray}
\k(t)\cdot\big{(}\k(t)-1\big{)}^{2}\cdot Y^{5}\big{(}\k(t)\big{)}=t^{5}.
\label{inverse}
\end{eqnarray}
Let $\alpha(t)$ be the inverse of $\k(t)$: $\k(\alpha(t))=t$. Substituting $\alpha(x)\mapsto t$ into the above, we get
\begin{eqnarray*}
x\cdot\big{(}x-1\big{)}^{2}\cdot Y^{5}(x)=\alpha^{5}(x).
\end{eqnarray*}
This and (\ref{invv}) gives
\begin{eqnarray*}
\alpha(x)=\frac{1}{5}
\int\limits_{1}^{\frac{1}{1-x}}\frac{\d t}{t^{3/5}(t-1)^{4/5}},\quad 0<x\leq 1,\quad
\alpha(1)=\frac{1}{5}B\Big{(}\frac{1}{5},\frac{2}{5}\Big{)}=
\frac{\Gamma(\frac{1}{5})\Gamma^{2}(\frac{2}{5})\sqrt{10+2\sqrt{5}}}{20\pi}.
\end{eqnarray*}
So, indeed, $\alpha(x)$ is an abelian integral, and $\k(x)$ is the inverse of it, an abelian function.\\

If we know this function, we can give immediately an analytic expression for $P$ and $Q$ (\cite{alkauskas-un}, Subsection 4.1):
\begin{eqnarray*}
P\Big{(}\frac{\k(a)\v}{a},\frac{\v}{a}\Big{)}\bl Q\Big{(}\frac{\k(a)\v}{a},\frac{\v}{a}\Big{)}=\frac{\k(a-\v)\v}{a-\v}\bl\frac{\v}{a-\v}.
\end{eqnarray*}
Now let us return to the flow $\Delta(\m{x})=\ell\circ\widetilde{\Delta}\circ\ell^{-1}(\m{x})$. First,
\begin{eqnarray*}
\widetilde{\Delta}\circ\ell^{-1}\Big{(}\frac{\k(a)\v}{a},\frac{\v}{a}\Big{)}=
\widetilde{\Delta}\Bigg{(}\frac{\k(a)\v}{aY\big{(}\k(a)\big{)}},
\frac{\v}{aY\big{(}\k(a)\big{)}}\Bigg{)}\\
=P\Big{(}\frac{\k(a)\tilde{\v}}{a},\frac{\tilde{\v}}{a}\Big{)}\bl 
Q\Big{(}\frac{\k(a)\tilde{\v}}{a},\frac{\tilde{\v}}{a}\Big{)}=\frac{\k(a-\tilde{\v})\tilde{\v}}{a-\tilde{\v}}\bl\frac{\tilde{\v}}{a-\tilde{\v}}.
\end{eqnarray*}
Here
\begin{eqnarray*}
\tilde{\v}=\frac{\v}{Y\big{(}\k(a)\big{)}}.
\end{eqnarray*}
Remember that from (\ref{inverse}) we have:
\begin{eqnarray*}
Y\big{(}\k(t)\big{)}=\frac{t}{\k^{1/5}(t)\big{(}\k(t)-1\big{)}^{2/5}}.
\end{eqnarray*}
So,
\begin{eqnarray}
\vartheta\Big{(}\frac{\k(a)\v}{a},\frac{\v}{a}\Big{)}\bl
\xi\Big{(}\frac{\k(a)\v}{a},\frac{\v}{a}\Big{)}=\ell\circ\widetilde{\Delta}\circ\ell^{-1}\Big{(}\frac{\k(a)\v}{a},\frac{\v}{a}\Big{)}
\nonumber\\
=\frac{\k(a-\tilde{\v})\tilde{\v}}{a-\tilde{\v}}Y\big{(}
\k(a-\tilde{\v})\big{)}\bl
\frac{\tilde{\v}}{a-\tilde{\v}}Y\big{(}\k(a-\tilde{\v})\big{)}\label{general}\\
=\frac{\k^{4/5}(a-\tilde{\v})\tilde{\v}}{\big{(}\k(a-\tilde{\v})-1\big{)}^{2/5}}\bl
\frac{\tilde{\v}}{\k^{1/5}(a-\tilde{\v})\big{(}\k(a-\tilde{\v})-1\big{)}^{2/5}}.
\label{super}
\end{eqnarray}
Let
\begin{eqnarray*}
x=\frac{\k(a)\v}{a},\quad y=\frac{\v}{a}.
\end{eqnarray*}
Then  
\begin{eqnarray*}
\k(a)=\frac{x}{y},\quad a=\alpha\Big{(}\frac{x}{y}\Big{)},\quad
\v=y\alpha\Big{(}\frac{x}{y}\Big{)},\quad
\tilde{\v}=\frac{y\alpha\Big{(}\frac{x}{y}\Big{)}}{Y\Big{(}\frac{x}{y}\Big{)}}=
x^{1/5}(x-y)^{2/5}y^{2/5}.
\end{eqnarray*}
This gives Proposition \ref{prop-abel}. We double-check that the boundary conditions for $\vartheta$ and $\xi$ are indeed satisfied. In particular, the expression (\ref{expann}) is the series for the abelian function
\begin{eqnarray*}
\frac{\vartheta(x,-x)}{\xi(x,-x)}=\k(c-\sqrt[5]{4}\cdot x),\quad c=\alpha(-1)=
-\frac{1}{5}\int\limits_{\frac{1}{2}}^{1}\frac{\d t}{t^{3/5}(t-1)^{4/5}}=-0.927975024086478_{+}.
\end{eqnarray*} 

Now we will derive power series expansion of $\k(x)$ around $c=\alpha(-1)$. Let
\begin{eqnarray*}
f(x)=\frac{1}{5x^{3/5}(x-1)^{4/5}},\quad F(x)=\int\limits_{1}^{x}f(t)\d t.
\end{eqnarray*}
Let us in the identity $F(\frac{1}{1-x})=\alpha(x)$ put $x\rightarrow\k(x)$, and differentiate it. We obtain
\begin{eqnarray}
\k'(x)=\frac{(1-\k(x))^{2}}{f\Big{(}\frac{1}{1-\k(x)}\Big{)}}\Rightarrow 
\k'(x)=5\big{(}1-\k(x)\big{)}^{3/5}\k^{4/5}(x).
\label{diff-k}
\end{eqnarray} 
For example, the property $\k(\alpha(-1))=-1$ gives $\k'(\alpha(-1))=5\sqrt[5]{8}$. Let us define the unknown coefficients $a_{i}$ by 
\begin{eqnarray*}
\k\big{(}\alpha(-1)-\sqrt[5]{4}x\big{)}=-1-\sum\limits_{i=1}^{\infty}a_{i}x^{i}=-1-L.
\end{eqnarray*} 
Then
\begin{eqnarray*}
\k'\big{(}\alpha(-1)-\sqrt[5]{4}x\big{)}=\frac{1}{\sqrt[5]{4}}\sum\limits_{i=1}^{\infty}ia_{i}x^{i-1}.
\end{eqnarray*}
Plugging this into (\ref{diff-k}), we obtain
\begin{eqnarray*}
\sum\limits_{i=1}^{\infty}ia_{i}x^{i-1}=
10\Big{(}1+\sum\limits_{i=1}^{\infty}\frac{a_{i}}{2}x^{i}\Big{)}^{3/5}\Big{(}1+\sum\limits_{i=1}^{\infty}a_{i}x^{i}\Big{)}^{4/5}.
\end{eqnarray*}
An expansion of the last two expressions in brackets via the binomial theorem gives the recursion for the coefficients $a_{i}$, it is uniquely solvable, and we readily see that they are rational numbers. We now double-check that indeed we obtain exactly the (negatives) of the coefficients in (\ref{expann}), which was calculated independently using (\ref{reku}). This very fast method which leads to the expansion (\ref{expann}) is thus very efficient to derive exact values for Taylor series expansion of an abelian function around any point at which its value is rational.\\

 Having verified the powers series expansion for $\frac{\vartheta(x,-x)}{\xi(x,-x)}$, let us calculate the series  of
\begin{eqnarray*}
\frac{(1+L)^{4/5}\sqrt[5]{4}}{(2+L)^{2/5}}=\frac{(1+L)^{4/5}}{(1+\frac{L}{2})^{2/5}}.
\end{eqnarray*} 
MAPLE returns exactly the series (\ref{theta-exp}), and this confirms once again the validity of Proposition \ref{prop-abel}.
\section{Integral non-abelian flow}
\label{sect-int}

\begin{figure}
\epsfig{file=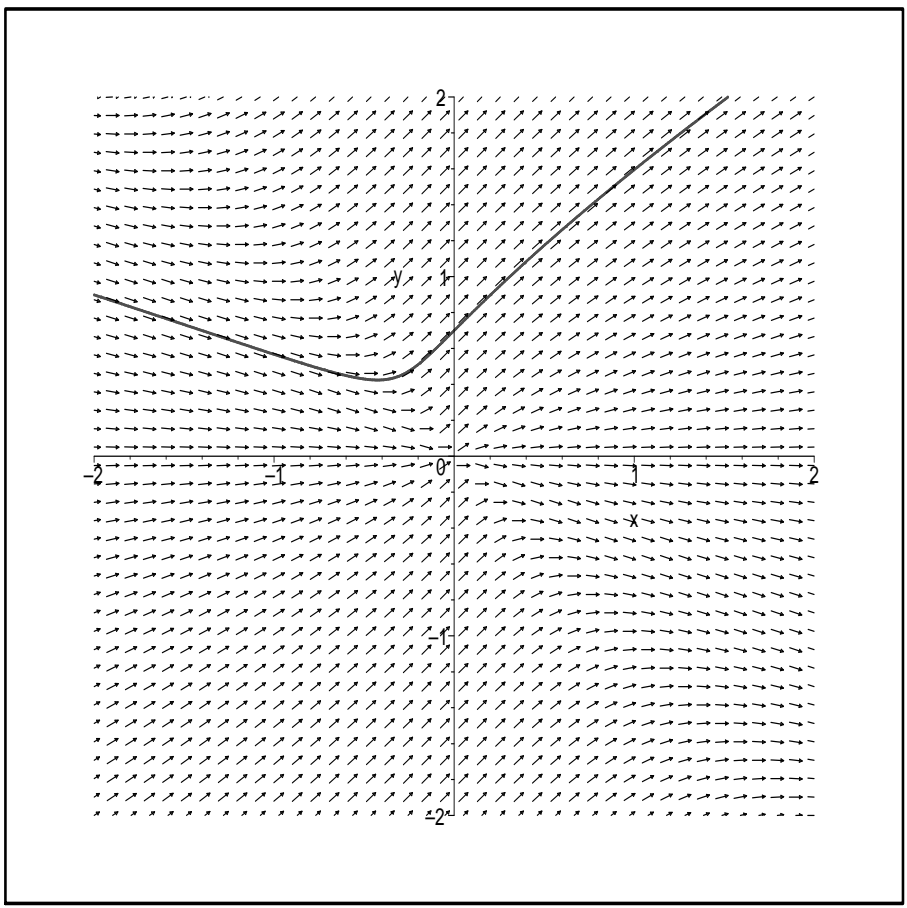,width=340pt,height=340pt,angle=-90}
\caption{The normalized vector field of the integral non-abelian flow $\Xi(\m{x})$, $(x,y)\in[-2,2]^{2}$ The orbit shown is 
$\exp(-\frac{x}{y}-\frac{x^2}{2y^2})y=0.7$}
\label{figure6}
\end{figure}

Let us look back at the Diagram in Figure \ref{diagram}. As we know \cite{alkauskas-un}, there are $4$ exceptional pairs of quadratic forms that produce unramified non-rational flows. Further, all pairs of quadratic forms (\ref{ova}) become integral after, if necessary, an additional conjugation with a linear map $(x,y)\mapsto(xz,yz)$.  As is clear from (\cite{alkauskas}, Subsection 4.3), and this can be easily proven using  (\ref{orbits}), the following proposition holds.

\begin{prop}Let $\varpi(x,y)\bl\varrho(x,y)$ is a pair of quadratic forms with integer coefficients. Then the orbits of this flow are algebraic curves if and only if $y\varpi(x,y)-x\varrho(x,y)$ has no multiple roots. If it has a multiple root, there exists a linear change, such that conjugating the vector field with it, we obtain one of the following pairs of quadratic forms:
\begin{itemize}
\item[i)] $x^2+xy+\lambda y^2\bl xy+y^2$, $\lambda\neq 0$ (a triple root).
\item[ii)] $xy\bl x^2+y^2$ (a triple root).
\item[iii)] $x^2+xy\bl \lambda xy+y^2$, $\lambda\neq 1$ (a double root).
\end{itemize}
\end{prop} 
As the most analytically complicated example, we will now integrate the vector field given in i) in the special case $\lambda=1$; for a general $\lambda\neq 0$ the results are completely analogous. Note that, differently from abelian (algebraic orbits) case, here the property that $\lambda$ is or is not an integer is of no importance. So, let
\begin{eqnarray*}
\varpi(x,y)\bl\varrho(x,y)=x^2+xy+y^2\bl xy+y^2.
\end{eqnarray*}

The method applied in Section \ref{abel-non-ell} is very general and is applicable here, only with all necessary adjustments. So we will not repeat all the arguments step by step, and only provide a sketch. The final verification based on power series expansion confirms that all the ends do meet.\\

Let this flow be $\Xi(\m{x})=G(x,y)\bl H(x,y)$. In particular, using the same method as in  the previous section, that is, the recursion (\ref{reku}), we can calculate:
\begin{eqnarray}
\frac{G(x,-x)}{x}
=1+x+\frac{1}{2}x^2+\frac{2}{3}x^3+\frac{7}{24}x^4+\frac{13}{30}x^5+\frac{127}{720}x^6+\frac{88}{315}x^7+\frac{4369}{40320}x^8\nonumber\\
+\frac{4069}{22680}x^9+\frac{34807}{518400}x^{10}+\frac{17926}{155925}x^{11}
+\frac{20036983}{479001600}x^{12}+\frac{7157977}{97297200}x^{13}+\cdots.
\label{g-expan}
\end{eqnarray}
We will soon see that the radius of convergence of this series is $\frac{\sqrt{\pi}}{\sqrt{2}}$, and this agrees numerically very well with the Cauchy's formula for the radius of convergence of power series.\\

The orbits of the flow $\Xi$ are found via (\ref{orbits}); so, they are given by
\begin{eqnarray*}
\mathscr{W}(x,y)=\exp\Bigg{(}-\frac{x}{y}-\frac{x^2}{2y^2}\Bigg{)}y=\mathrm{const.}
\end{eqnarray*}
The solution of the differential equation (\ref{33}), if we choose the $``+"$ sign, is given by $Z(x)$, where
\begin{eqnarray*}
Z(x)=-\frac{\sqrt{\pi e}}{\sqrt{2}}\exp\Bigg{(}x+\frac{x^2}{2}\Bigg{)}\cdot
\mathrm{erf}\Bigg{(}\frac{x+1}{\sqrt{2}}\Bigg{)},\text{ where }
\mathrm{erf}(x)=\frac{2}{\sqrt{\pi}}\int\limits_{0}^{x}e^{-t^2}\d t.
\end{eqnarray*} 
So, $Z(x)$ is an entire function. Let us define the function $\l(t)$ implicitly from
\begin{eqnarray}
Z(\l(t))\cdot\exp\Big{(}-\l(t)-\frac{\l^{2}(t)}{2}\Big{)}=t.
\label{inverse2}
\end{eqnarray}
Let $\beta(t)$ be the inverse of $\l(t)$. Then
\begin{eqnarray*}
\beta(t)=-\frac{\sqrt{\pi e}}{\sqrt{2}}\mathrm{erf}\Bigg{(}\frac{t+1}{\sqrt{2}}\Bigg{)}.
\end{eqnarray*} 
So,
\begin{eqnarray*}
\l(t)=\sqrt{2}\mathrm{erf}^{-1}\Bigg{(}-\frac{t\sqrt{2}}{\sqrt{\pi e}}\Bigg{)}-1.
\end{eqnarray*}
For example, for real $t$, we have
\begin{eqnarray*}
\l(0)=-1,\quad \l\Big{(}\frac{\sqrt{\pi e}}{\sqrt{2}}\Big{)}=-\infty,\quad 
\l\Big{(}-\frac{\sqrt{\pi e}}{\sqrt{2}}\Big{)}=\infty.
\end{eqnarray*}
In particular, using the known power series expansion of the $\mathrm{erf}^{-1}$ function (see ``The Online Encyclopedia of Integer Sequences", A092676 and A092677 \cite{oeis}), we have:\small
\begin{eqnarray}
\l(xe^{1/2})=-1-x-\frac{1}{6}x^{3}-\frac{7}{120}x^{5}
-\frac{127}{5040}x^{7}-\frac{4369}{362880}x^{9}-\frac{34807}{5702400}x^{11}-\frac{20036983}{6227020800}x^{13}
-\cdots.
\label{erfinv}
\end{eqnarray} \normalsize
Next, the general formula (\ref{general}), valid for all flows, claims that
\begin{eqnarray*}
G\Big{(}\frac{\l(a)\v}{a},\frac{\v}{a}\Big{)}\bl
H\Big{(}\frac{\l(a)\v}{a},\frac{\v}{a}\Big{)}
=\frac{\l(a-\tilde{\v})\tilde{\v}}{a-\tilde{\v}}Z\big{(}
\l(a-\tilde{\v})\big{)}\bl
\frac{\tilde{\v}}{a-\tilde{\v}}Z\big{(}\l(a-\tilde{\v})\big{)},
\end{eqnarray*}
where
\begin{eqnarray*}
\tilde{\v}=\frac{\v}{Z(\l(a))}.
\end{eqnarray*}
Remember that from (\ref{inverse2}) we have:
\begin{eqnarray*}
Z\big{(}\l(t)\big{)}=t\exp\Big{(}\l(t)+\frac{\l^{2}(t)}{2}\Big{)}.
\end{eqnarray*}
Let
\begin{eqnarray*}
x=\frac{\l(a)\v}{a},\quad y=\frac{\v}{a}.
\end{eqnarray*}
Then  
\begin{eqnarray*}
\l(a)=\frac{x}{y},\quad a=\beta\Big{(}\frac{x}{y}\Big{)},\quad
\v=y\beta\Big{(}\frac{x}{y}\Big{)},\quad
\tilde{\v}=\frac{y\beta\Big{(}\frac{x}{y}\Big{)}}{Z\Big{(}\frac{x}{y}\Big{)}}=
\exp\Bigg{(}-\frac{x}{y}-\frac{x^2}{2y^2}\Bigg{)}y.
\end{eqnarray*}
So, we obtain
\begin{eqnarray*}
G(x,y)\bl H(x,y)=\psi\v\exp\Big{(}\psi+\frac{\psi^{2}}{2}\Big{)}\bl \v\exp\Big{(}\psi+\frac{\psi^{2}}{2}\Big{)},\\
\text{ where } 
\psi=\l\Big{(}\beta\Big{(}\frac{x}{y}\Big{)}-\v\Big{)},
\v=\exp\Bigg{(}-\frac{x}{y}-\frac{x^2}{2y^2}\Bigg{)}y.
\end{eqnarray*}
This gives Proposition \ref{prop-int}. In the special case $(x,y)=(x,-x)$, we have:
\begin{eqnarray*}
\v=-xe^{1/2},\quad\beta\Big{(}\frac{x}{y}\Big{)}=\beta(-1)=0,\quad\psi=\l(xe^{1/2}).
\end{eqnarray*}
So,
\begin{eqnarray*}
\frac{G(x,-x)}{x}=-e^{1/2}\l(xe^{1/2})\exp\Big{(}\l(xe^{1/2})+\frac{1}{2}\l^{2}(xe^{1/2})\Big{)}
=-\l(xe^{1/2})\exp\Big{(}\frac{1}{2}\big{(}\l(xe^{1/2})+1\big{)}^2\Big{)}.
\end{eqnarray*}
Now, the power series for $\l(xe^{1/2})+1$ starts at $-x$, so if we substitute (\ref{erfinv}) into the above, the power series for $\exp$, we obtain again a power series, and each coefficient is a finite sum. We plug thus (\ref{erfinv}), and MAPLE gives back what is expected, the powers series (\ref{g-expan}). So, we can be sure that all the reasoning is correct!

\section{Classification of algebraic flows}
Now we are in a position to prove Theorem \ref{thm3}. Since algebraic flow is necessarily an abelian flow, we are in a situation desribed by Theorem \ref{thm2}. First, we consider the case I.

\subsection{Abelian flows of type I which are algebraic} Let us consider (\ref{ovaa}):
\begin{eqnarray*}
\varpi(x,y)\bl\varrho(x,y)=\frac{B-1}{C-1}x^{2}+Bxy\bl
\frac{B-1}{C-1}Cxy+y^2.
\end{eqnarray*} 
Or task to to find all pairs of rational numbers $B,C$ such that this vector field gives rise to an algebraic flow. For simplicity, put
\begin{eqnarray*}
\frac{C}{C-1}=-n,\quad C=\frac{n}{n+1},\quad 1-B=Q.
\end{eqnarray*}
So, we are dealing with the vector field
\begin{eqnarray}
\varpi(x,y)\bl\varrho(x,y)=Q(n+1)x^{2}+(1-Q)xy\bl
Qnxy+y^2,\quad Q,n\in\mathbb{Q}.
\label{vf}
\end{eqnarray} 
The case $Q=0$ gives the flow $\frac{x}{1-y}\bl\frac{y}{1-y}$, and the case $n=0$ produces indeed algebraic flow (the explicit formulas will be soon given), so we will assume henceforth $n,Q \neq 0$.
We will now integrate this vector field. Since $x\varrho(x,1)-\varpi(x,1)=Qx(1-x)$, the differential equation (\ref{33}), in case we choose the ``+" sign, reads as
\begin{eqnarray}
Qx(1-x)f'(x)+(Qnx+1)f(x)=1.
\label{sch}
\end{eqnarray}
The solution of this differential equation can be given by 
\begin{eqnarray*}
\,_{2}F_{1}\Big{(}-n,1;1+\frac{1}{Q};x\Big{)}:=P_{n,Q}(x).
\end{eqnarray*}

The orbits of the flow $\phi(\m{x})=u(x,y)\bl v(x,y)$ with the vector field $\varpi(x,y)\bl \varrho(x,y)$ is given by
\begin{eqnarray*}
\mathscr{W}(x,y)=x^{\frac{1}{Q}}(x-y)^{-n-\frac{1}{Q}}y^{n+1}=\mathrm{const.}
\end{eqnarray*}
 Let us define the function $\n(t)$ implicitly from
\begin{eqnarray}
P_{n,Q}(\n(t))\cdot \big{(}\n(t)\big{)}^{\frac{1}{Q}}\big{(}\n(t)-1\big{)}^{-n-\frac{1}{Q}}=t.
\label{inverse3}
\end{eqnarray}
Let $\gamma(t)$ be the inverse of $\n(t)$. Then
\begin{eqnarray*}
\gamma(t)=P_{n,Q}(t)\cdot t^{\frac{1}{Q}}(t-1)^{-n-\frac{1}{Q}}.
\end{eqnarray*} 

Next, the general formula (\ref{general}) gives
\begin{eqnarray*}
u\Big{(}\frac{\n(a)\v}{a},\frac{\v}{a}\Big{)}\bl
v\Big{(}\frac{\n(a)\v}{a},\frac{\v}{a}\Big{)}
=\frac{\n(a-\tilde{\v})\tilde{\v}}{a-\tilde{\v}}P_{n,Q}\big{(}
\n(a-\tilde{\v})\big{)}\bl
\frac{\tilde{\v}}{a-\tilde{\v}}P_{n,Q}\big{(}\n(a-\tilde{\v})\big{)},
\end{eqnarray*}
where
\begin{eqnarray*}
\tilde{\v}=\frac{\v}{P_{n,Q}(\n(a))}.
\end{eqnarray*}
Let, as before,
\begin{eqnarray*}
x=\frac{\n(a)\v}{a},\quad y=\frac{\v}{a}.
\end{eqnarray*}
Then  
\begin{eqnarray*}
\n(a)=\frac{x}{y},\quad a=\gamma\Big{(}\frac{x}{y}\Big{)},\quad
\v=y\gamma\Big{(}\frac{x}{y}\Big{)},\quad
\tilde{\v}=\frac{y\gamma\Big{(}\frac{x}{y}\Big{)}}{P_{n,Q}\Big{(}\frac{x}{y}\Big{)}}=
x^{\frac{1}{Q}}(x-y)^{-n-\frac{1}{Q}}y^{n+1}.
\end{eqnarray*}
Note that similarly as in Sections \ref{abel-non-ell} and \ref{sect-int}, $\v$ is necessarily $1-$homogenic.
So,
\begin{eqnarray}
u(x,y)\bl v(x,y)&=&\n^{1-\frac{1}{Q}}\big{(}\varkappa\big{)}
\Big{(}\n(\varkappa)-1\Big{)}^{n+\frac{1}{Q}}\v\bl 
\n^{-\frac{1}{Q}}\big{(}\varkappa\big{)}
\Big{(}\n(\varkappa)-1\Big{)}^{n+\frac{1}{Q}}\v,
\label{alg-expli}\\
\text{where }\varkappa&=&\gamma\Big{(}\frac{x}{y}\Big{)}-\v,\v=x^{\frac{1}{Q}}(x-y)^{-n-\frac{1}{Q}}y^{n+1}.
\nonumber
\end{eqnarray}
We can simplify this. Note that from the orbits property,
\begin{eqnarray}
u^{\frac{1}{Q}}(u-v)^{-n-\frac{1}{Q}}v^{n+1}=x^{\frac{1}{Q}}(x-y)^{-n-\frac{1}{Q}}y^{n+1}.
\label{flow-alg}
\end{eqnarray}
Also, since $\frac{u(x,y)}{v(x,y)}=\n(\varkappa)$, we obtain
\begin{eqnarray*}
\gamma\Big{(}\frac{u}{v}\Big{)}=\varkappa=\gamma\Big{(}\frac{x}{y}\Big{)}-x^{\frac{1}{Q}}(x-y)^{-n-\frac{1}{Q}}y^{n+1}.
\end{eqnarray*}
This, after some calculations, gives the identity
\begin{eqnarray}
\frac{1}{v}P_{n,Q}\Big{(}\frac{u}{v}\Big{)}=\frac{1}{y}P_{n,Q}\Big{(}\frac{x}{y}\Big{)}-1.
\label{tarp}
\end{eqnarray}
Together with an orbit conservation property (\ref{flow-alg}), this defines the pair of functions $(u,v)$. We are left to explore for which pairs of rational numbers $(n,Q)$ these two equations have algebraic functions as solutions. We will show that this happens only if $P_{n,Q}$ is an algebraic function. On the other hand, if this holds, then $u$ and $v$ are algebraic functions too, since they satisfy a pair of independent algebraic equations.\\

So, suppose $u\bl v$ is a flow and $u,v$ are algebraic functions. Fix $x,y\neq 0$, $x\neq y$, and consider the line $(x\omega,y\omega )$ for $\omega$ being a variable. We claim that then
\begin{eqnarray*}
\frac{u(x\omega,y\omega)}{v(x\omega,y\omega)}
\end{eqnarray*} 
is a non-constant algebraic function in $\omega$. Indeed, suppose it is constant. Then the orbit conservation property (\ref{flow-alg}) gives
\begin{eqnarray*}
v(x\omega,y\omega)\Big{(}\frac{u}{v}\Big{)}^{\frac{1}{Q}}
\Big{(}\frac{u}{v}-1\Big{)}^{-n-\frac{1}{Q}}=
y\omega\Big{(}\frac{x}{y}\Big{)}^{\frac{1}{Q}}
\Big{(}\frac{x}{y}-1\Big{)}^{-n-\frac{1}{Q}}.
\end{eqnarray*}
So, for variable $\omega$, we see that $\frac{u}{v}\neq 0,1$, and moreover,
\begin{eqnarray*}
v(x\omega,y\omega)=y\omega\delta,
\end{eqnarray*}
for a certain constant $\delta$. This gives $u(x\omega,y\omega)=y\omega\gamma$, for another constant $\gamma$. Then the boundary conditions (\ref{bound}) imply the exact values for $\delta$ and $\gamma$:
\begin{eqnarray*}
u(x\omega,y\omega)=x\omega,\quad v(x\omega,y\omega)=y\omega.
\end{eqnarray*}
But now this contradicts to (\ref{tarp}) for $\omega$ small enough (to avoid questions of ramification).\\

So, we have proved that $\frac{u}{v}$ on the line $(x\omega,y\omega)$ is not a constant function. This gives
\begin{eqnarray*}
P_{n,Q}\Big{(}\frac{u(x\omega,y\omega)}{v(x\omega,y\omega)}\Big{)}=\frac{v(x\omega,y\omega)}{y\omega}P_{n,Q}\Big{(}\frac{x}{y}\Big{)}-v(x\omega,y\omega).
\end{eqnarray*}
As a function in $\omega$ (the numbers $x,y$ being fixed), the right hand side is an algebraic function. On the other hand, the argument of $P_{n,Q}$ on the  left is a non-constant algebraic function. If $P_{n,Q}$ has an infinite monodromy, the left hand side has an infinite monodromy, too - a contradiction. Thus, the only possibility is when $P_{n,Q}$ itself is in fact \emph{an algebraic function}. \\

H. Schwarz \cite{vidunas} classified all cases when a hypergeometric function has a finite monodromy and thus is an algebraic function. This research was greatly generalized by F. Beukers in relation to higher hypergeometric functions \cite{beukers}. The local exponent differences (Schwarz parameters, see \cite{vidunas}) at three regular singular points $0$, $1$ and $\infty$ are then (up to a sign) in our case are equal to 
\begin{eqnarray}
\Big{(}-\frac{1}{Q},n+\frac{1}{Q},-n-1\Big{)}.
\label{param-s}
\end{eqnarray}
We immediately recognize the (negatives) of the three numbers (\ref{cyk}) that add up to $1$. Now, going through H. Schwarz's list to see when the function $P_{n,Q}(x)$ is algebraic with the condition that Schwarz parameters (\ref{param-s}) add up to an integer, we see that the only possible case is when one of the parameters is a negative integer, say, such is $-n-1$, $n\in\mathbb{N}_{0}$, and $Q$ is an arbitrary rational number, except when $\frac{1}{Q}\in\{-n,-n+1,\ldots,-1\}$, in which case $P_{n,Q}$ is not an algebraic function. In fact, we can manage without Schwarz's list. Namely, the differential equation (\ref{sch}) is easily solved in quadratures. The solution is algebraic function if and only if
\begin{eqnarray*}
\int(x-1)^{-n-1-\frac{1}{Q}}x^{-1+\frac{1}{Q}}\d x
\end{eqnarray*} 
is algebraic. Now, the answer follows from calculus routines: we can rotate all three Schwarz parameters, which corresponds to a change of variable in the above integral. This gives Theorem \ref{thm3}.
\subsection{Abelian flows of type II which are algebraic}
\label{sub-alg-II}
As we know from (\cite{alkauskas}, Section 4), we cannot further reduce the pair of $2-$homogenic rational functions $\varpi(x,y)\bl\varrho(x,y)$ with a help of conjugating this flow with a $1-$BIR in case $\varrho(x,y)=0$ and $\varpi(x,y)$ is a $2-$homogenic rational function whose numerator has no multiple roots. This is \emph{an obstruction} case (see \cite{alkauskas}, Section 4, case II). So, henceforth in this subsection we assume that this case holds, and investigate when this does lead to an algebraic flow $u(x,y)\bl v(x,y)$ - it will appear, never. We know, of course, that in this case $v(x,y)=y$, the orbits of the flow are given by $y=\mathrm{const.}$, and $u(x,y)$ is defined implicitly from the identity (\cite{alkauskas}, Eq. (60))
\begin{eqnarray}
\int\limits_{y/u(x,y)}^{y/x}\frac{\d t}{\varpi(1,t)}=y.
\label{inte-II}
\end{eqnarray} 
Now, since the numerator of $\varpi(1,t)$ has no multiple roots, we can write
\begin{eqnarray*}
\frac{1}{\varpi(1,t)}=P(t)+\sum\limits_{i=1}^{r}\frac{a_{i}}{t-\alpha_{i}},\quad \alpha_{j}\neq\alpha_{k}\text{ for }j\neq k,
\end{eqnarray*}
where $P(t)$ is a polynomial. Let $\widehat{P}$ be the anti-derivative of $P$. Integration of (\ref{inte-II}) gives
\begin{eqnarray*}
\widehat{P}\Big{(}\frac{y}{x}\Big{)}-\widehat{P}\Big{(}\frac{y}{u(x,y)}\Big{)}-y=-\sum\limits_{i=1}^{\infty}a_{i}\ln
\Bigg{(}\frac{\frac{y}{x}-\alpha_{i}}{\frac{y}{u(x,y)}-\alpha_{i}}\Bigg{)}.
\end{eqnarray*}
This is satisfied by an algebraic function $u(x,y)$ and all branches of the logarithm. However, for every algebraic function $u(x,y)$ the left hand side has a finite monodromy group (for every pair $x\bl y$ it attains a finite number of values), while the right hand side has an infinite monodromy unless all $a_{i}=0$. Suppose then the latter. Yet, in this case  we have
\begin{eqnarray*}
\varpi(x,y)=\frac{x^2}{P(\frac{y}{x})},
\end{eqnarray*}
and the numerator of $\varpi$ is divisible by $x^{2+\mathrm{deg}(P)}$, which contradicts the condition that it has no multiple roots. Thus, there are no algebraic flows which are abelian flows of type II.

\appendix\section{Addendum to ``The projective translation equation and unramified $2$-dimensional flows with rational vector fields"}
\label{addendum}

The names of the Theorems and Sections in this Addendum refer to \cite{alkauskas-un}, unless otherwise stated. The references to displayed equation, like (\ref{g-parteq}), refer to the current paper only.\\

 In this Addendum we will expand Subsection 4.5 and clarify a bit Theorem 1. Here we use notations of the latter Theorem. Let $\v=[xy(x-y)]^{1/3}$. The flow $\Lambda(\m{x})=\lambda(x,y)\bl\lambda(y,x)$ is unramified. Indeed, as is clear from the explicit formulas in Theorem 2, $\cm(\v)$ and $\sm(\v)\v^{-1}$ contain only integral powers of $x$ and $y$. Let now $\v=[xy(x-y)^2]^{1/4}$. The same can be said about the flow $\Psi(\m{x})=\psi(x,y)\bl\psi(y,x)$ and the two elliptic functions $p(u)$ and $q(u)$, which were calculated in Subsection 4.5. Namely, the functions $p(\v)\v$ and $q(\v)\v$ contain only integral powers of $x$ and $y$. However, the situation with the vector field $x^2-xy\bl y^2-2xy$  is a bit more interesting. This vector field is briefly treated in the last 5 lines of Subsection 4.5. In order to fully prove Theorem 1, here we expand on this vector field.\\
  
So, let
\begin{eqnarray*}
\varpi\bl\varrho=x^2-xy\bl y^2-2xy, \quad \v=[x^3y^2(3x-2y)]^{1/6}.
\end{eqnarray*}
As we know, the flow $\alpha(x,y)\bl\beta(x,y)$, which is generated by this vector field, can be integrated explicitly in terms of Dixonian elliptic functions. Indeed, to show this, let $p$ and $q$ be two functions which satisfy the following properties. The Taylor series at $u=0$ of $p(u)$ and $q(u)$ start, respectively, at powers $u^{-1}$ and $u^2$, and 
\begin{eqnarray}
\left\{\begin{array}{c@{\qquad}c}
1\equiv p^{3}q^2(3p-2q),\\
p'=-p^{2}+pq,\\
q'=-q^{2}+2pq.
\end{array}\right.\label{p-q}
\end{eqnarray} 
Using these properties, one can recurrently calculate the unique Taylor coefficients, and we do it with the help of MAPLE:\tiny
\begin{eqnarray*}
p(u)=u^{-1}+\frac{\sqrt{3}}{12}u^{2}+\frac{1}{1008}u^{5}
+\frac{\sqrt{3}}{36288}u^{8}+\frac{11}{39626496}u^{11}
+\frac{\sqrt{3}}{118879488}u^{14}+\frac{193}{2276779954176}u^{17}
+\cdots,\\
q(u)=\frac{\sqrt{3}}{3}u^{2}-\frac{1}{18}u^{5}+\frac{43\sqrt{3}}{9072}u^{8}-\frac{31}{27216}u^{11}
+\frac{2731\sqrt{3}}{29719872}u^{14}
-\frac{47545}{2139830784}u^{17}+\frac{873815\sqrt{3}}{487881418752}u^{20}+\cdots.
\end{eqnarray*}\normalsize
Using induction we see that coefficients at $u^{6n-1}$, $n\in\mathbb{N}_{0}$, are rational numbers, and the coefficients at $u^{6n+2}$ are of the form $r\sqrt{3}$, $r\in\mathbb{Q}$. We know that the formulas for $\alpha$ and $\beta$ will contain $p(\v)\v$ and $q(\v)\v$; see Subsection 4.1. So, differently from the analysis of the flows $\Lambda$ and $\Psi$, $\alpha$ and $\beta$ will potentially contain a quadratically ramified function $\sqrt{x^3y^2(3x-2y)}$. However, note that if $(p(u),q(u))$ solves (\ref{p-q}), so does the pair of functions $(-p(-u),-q(-u))$. In case of the vector field $x^2-3xy\bl y^2-3xy$ and the flow $\Psi$, this transform is an identity map (see the formulas on the top of p. 903), since both elliptic functions involved are odd functions. Further, in case of the vector field $x^2-2xy\bl y^2-2xy$ this does not work, since the orbits $xy(x-y)$ are of degree $3$, which is an odd number, and so the pair $(-\s(-u),-\c(-u))$ parametrizes the curve $xy(x-y)=-1$ rather that the curve $xy(x-y)=1$. So, only in the case of the vector field $x^2-xy\bl y^2-2xy$ we get an interesting non-trivial situation. For example, this shows that in the the Taylor series for $p(u)$ and $q(u)$ we can assume that the symbol $``\sqrt{3}"$ stands for a positive as well as for a negative value of the square root; this does not affect the validity of (\ref{p-q}). Further, we know that the Talor series for $\alpha(x,y)$ and $\beta(x,y)$ at the origin contain only rational numbers as coefficients; see (\ref{reku}). So, $\sqrt{3}$ does not appear at all. Thus, these formulas are stable under the automorphism of $\mathbb{Q}(\sqrt{3})$, and this will show that in fact in the final formulas, there is no quadratic ramification remaining, and the flow $\alpha\bl\beta$ is indeed unramified!\\
    
We will express these (elliptic) functions $p$ and $q$ in terms of Dixonian elliptic functions $\sm$ and $\cm$. Note that the formulas on (p. 883, \cite{alkauskas-un}) give\small
\begin{eqnarray}
\s(u)=-u^2-\frac{1}{28}u^{8}-\frac{3}{2548}u^{14}-\frac{15}{387296}u^{20}-\frac{449}{352439360}u^{26}-\frac{51279}{1223669457920}u^{32}+\cdots,\label{sp}\\
\c(u)=u^{-1}-\frac{1}{2}u^2+\frac{3}{28}u^{5}-\frac{1}{56}u^{8}+\frac{33}{10192}u^{11}-\frac{3}{5096}u^{14}+\frac{579}{5422144}u^{17}-\frac{15}{774592}u^{20}+\cdots.\nonumber
\end{eqnarray}\normalsize
As we know from Proposition 3, $\s(u)$ is an even function, so the Taylor series of $\s$ contains only the powers $u^{6n+2}$, $n\in\mathbb{N}_{0}$. This already sheds a light on the phenomenon of cancellation of $\sqrt{3}$ just mentioned. Indeed, the Taylor series for $\sm$ and $\cm$ contain the powers $u^{3n+1}$ and $u^{3n}$, respectivelly. So, $\s(u)=-\frac{\sm^2(u)}{\cm(u)}$ potentially contain all positive powers $u^{3n+2}$. But, as we see, the coeffiecients at $u^{6n+5}$ cancell out. Further, inspecting the coefficients of $\c$ at $u^{6n+2}$ we see that they are equal to a half of the corresponding Taylor coefficients of $\s(u)$. Indeed, we amend the cited Proposition 3 with the following fact.
\begin{prop}
The function $2\c(u)-\s(u)$ is an odd function.
\end{prop}
\begin{proof}
Indeed, according to Proposition 1,
\begin{eqnarray*}
2\c(-u)-\s(-u)=\frac{2}{\s(u)\c(u)}-\s(u).
\end{eqnarray*}
Thus, we need to show that
\begin{eqnarray*}
\frac{2}{\s(u)\c(u)}-\s(u)=-2\c(u)+\s(u).
\end{eqnarray*}
This simplifies as $\s(u)\c(u)[\s(u)-\c(u)]\equiv 1$, which, as we know, does hold.
\end{proof}

Now, consider the following birational transformation
\begin{eqnarray}
p=-\frac{\delta+\gamma}{2},\quad q=\frac{4\sqrt{3}}{3(\gamma^{2}-\delta^2)}.\label{bir}
\end{eqnarray}
The inverse is
\begin{eqnarray}
\delta=\frac{\sqrt{3}}{3pq}-p,\quad \gamma=-\frac{\sqrt{3}}{3pq}-p.\label{pq-inverse}
\end{eqnarray}
A direct calculation shows that this transforms our elliptic curve $p^3q^2(3p-2q)=1$ into
\begin{eqnarray}
\delta\gamma(\delta-\gamma)=\frac{4\sqrt{3}}{9}.
\label{cubic}
\end{eqnarray}
Using the Taylor series for $p$, $q$ and the transformation (\ref{pq-inverse}), we can calculate the Taylor series at the origin of $\delta$ and $\gamma$: \tiny
\begin{eqnarray*}
\delta=-\frac{\sqrt{3}}{9}u^2-\frac{\sqrt{3}}{27216}u^8
-\frac{\sqrt{3}}{89159616}u^{14}-\frac{5\sqrt{3}}{1463644256256}u^{20}-\frac{449\sqrt{3}}{431540872514519040}
u^{26}+\cdots,\\
\gamma=-2u^{-1}-\frac{\sqrt{3}}{18}u^2-\frac{1}{504}u^5-\frac{\sqrt{3}}{54432}u^8-\frac{11}{19813248}u^{11}
-\frac{\sqrt{3}}{178319232}u^{14}-\frac{193}{1138389977088}u^{17}+\cdots.
\end{eqnarray*}\normalsize
Compare these expansions to the expansions of $\s(u)$ and $\c(u)$. The coefficient at $u^{26}$ of both $\s$ and $\delta$ have the prime number $449$ in the numerator, while the coefficients of $\c$ and $\gamma$ at $u^{17}$ have the prime number $193$ in the numerator: $579=193\cdot 3$. Also, $(\alpha,\beta)$ parametrizes the same elliptic curve as $(\s,\c)$, only the scaled one. This suggests the following proposition.
\begin{prop}The elliptic functions $\alpha$ and $\beta$ are given by 
\begin{eqnarray*}
\delta=\frac{\sqrt[3]{4}}{\sqrt{3}}\s\Big{(}-\frac{u}{\sqrt[3]{2}\sqrt{3}}\Big{)},\quad
\gamma=\frac{\sqrt[3]{4}}{\sqrt{3}}\c\Big{(}-\frac{u}{\sqrt[3]{2}\sqrt{3}}\Big{)}.
\end{eqnarray*}
\end{prop}
\begin{proof}
Assume that we know in advance that
\begin{eqnarray}
\delta(u)=A\s(Bu),\quad \gamma(u)=A\c(Bu).\label{defini}
\end{eqnarray}
Comparing the leading coeffcients of all for functions involved, we get
\begin{eqnarray*}
\frac{\sqrt{3}}{9}=AB^{2},\quad -2=\frac{A}{B}.
\end{eqnarray*}
This gives
\begin{eqnarray*}
B=-\frac{1}{\sqrt[3]{2}\sqrt{3}},\quad A=\frac{\sqrt[3]{4}}{\sqrt{3}}.
\end{eqnarray*}
Knowing this, let us \emph{define} the functions $\delta$ and $\gamma$ by (\ref{defini}), and the let us \emph{define} the functions $p$ and $q$ by (\ref{bir}). Using the known properties of $\s$ and $\c$ we readily verify all the three properties of $p$ and $q$ given by (\ref{p-q}).
\end{proof}
So, we can express our functions $p$ and $q$ in terms of $\sm$ and $\cm$ as follows (all the arguments of $\s,\c,\sm,\cm$ are equal to $Bu$).
\begin{eqnarray*}
p(u)&=&-\frac{A\s+A\c}{2}=\frac{A\sm^{2}}{2\cm}-\frac{A\cm^{2}}{2\sm}=
\frac{A\sm^{3}-A\cm^{3}}{2\sm\cm}=\frac{2A\sm^{3}-A}{2\sm\cm},\\
q(u)&=&\frac{4\sqrt{3}}{3[A^2\c^2-A^2\s^2]}=
\frac{4\sqrt{3}}{3A^2[\frac{\cm^4}{\sm^2}-\frac{\sm^4}{\cm^2}]}
=\frac{4\sqrt{3}\sm^2\cm^2}{3A^2[\cm^6-\sm^6]}=
\frac{4\sqrt{3}\sm^2\cm^2}{3A^2[1-2\sm^3]}.
\end{eqnarray*}
In this form, the properties (\ref{p-q}) are easily verified by hand, using the properties of $\sm$ and $\cm$ given by Proposition 9.\\

Now, the pair $(\s,\c)$ is just a scaling of $(\delta,\gamma)$, and the latter, by the properties (\ref{bir}) and (\ref{pq-inverse}), is a birational transform of $(p,q)$. Since, according to Proposition 3, we have rational addition formulas for $\s(u-v)$ and $\c(u-v)$, this implies rational addition formulas for $p$ and $q$.\\

Now we are ready to explicitely integrate the vector field $\varpi\bl\varrho =x^2-xy\bl y^2-2xy$. Let, as already defined in (\cite{alkauskas-un}, Theorem 1), $\alpha\bl\beta$ be this flow.
Let $x=p(u)\v$, $y=q(u)\v$, $\v=[x^3y^2(3x-2y)]^{1/6}$. Similarly as in the (\cite{alkauskas-un}, Subsection 4.1), we find that the solution of the differential equation (\ref{g-parteq}) and the boudnary condition (\ref{bound}) in case of the vector field $x^2-xy\bl y^2-2xy$ satisfies
\begin{eqnarray*}
\alpha(x,y)=\alpha\Big{(}p(u)\v,q(u)\v\Big{)}=p(u-\v)\v,\\
\beta(x,y)=\beta\Big{(}p(u)\v,q(u)\v\Big{)}=q(u-\v)\v.
\end{eqnarray*}
So,
\begin{eqnarray*}
\alpha(x,y)=-\frac{\v}{2}\Big{(}\delta(u-\v)+\gamma(u-\v)\Big{)}=-\frac{A\v}{2}
\s(Bu-B\v)-\frac{A\v}{2}\c(Bu-B\v).
\end{eqnarray*}
Further, we use addition formulas for the elliptic functions $\s$ and $\c$, as given by Proposition 3. Thus,\small
\begin{eqnarray}
\alpha=
-\frac{A\v(C_{1}+C_{2}-S_{1}S_{2}C_{1}C_{2})^{2}S_{1}S_{2}}
{2(1-S_{1}^{2}S_{2}^{2}C_{1}C_{2})(S_{1}S_{2}C_{1}+S_{1}S_{2}C_{2}-1)}
-\frac{A\v(S_{1}S_{2}C_{1}+S_{1}S_{2}C_{2}-1)^{2}C_{1}C_{2}}
{2(1-S_{1}^{2}S_{2}^{2}C_{1}C_{2})(C_{1}+C_{2}-S_{1}S_{2}C_{1}C_{2})},\label{alpha}
\end{eqnarray}
\normalsize
where
\begin{eqnarray*}
S_{1}=\s(Bu),\quad C_{1}=\c(Bu),\quad 
S_{2}=\s(-B\v),\quad C_{2}=\c(-B\v).
\end{eqnarray*}
Now, $S_{1}$ and $C_{1}$ are just rational functions in $p(u)=x\v^{-1}$ and $q(u)=y\v^{-1}$:
\begin{eqnarray*}
S_{1}=\s(Bu)&=&\frac{\delta(u)}{A}=\frac{\sqrt{3}}{3Ap(u)q(u)}-\frac{p(u)}{A}=
\frac{\sqrt{3}\v^{2}}{3Axy}-\frac{x}{A\v},\\
C_{1}=\c(Bu)&=&\frac{\gamma(u)}{A}=-\frac{\sqrt{3}}{3Ap(u)q(u)}-\frac{p(u)}{A}=-\frac{\sqrt{3}\v^{2}}{3Axy}-\frac{x}{A\v}.
\end{eqnarray*}
Note that in the last expression, $\v$ appears with exponents $2$ and $-1$, and $2-(-1)=3$; this is important. Now, plug the last two expressions into (\ref{alpha}). A careful inspection of all the terms shows that $\alpha(x,y)$ is a rational function in $x,y$, $\v\s(-B\v)$, $\v^{-2}\s(-B\v)$, $\v\c(-B\v)$, and $\v^{-2}(-B\v)$. Indeed, in (\ref{alpha}) consider, for example,  the first summand, and the term of the numerator which is equal to $A\v C_{1}^{2}S_{1}S_{2}$. For it, we have
\begin{eqnarray*}
A\v C_{1}^{2}S_{1}S_{2}&=&A\v\Big{(}\frac{\sqrt{3}\v^{2}}{3Axy}+\frac{x}{A\v}\Big{)}^2\cdot\Big{(}\frac{\sqrt{3}\v^{2}}{3Axy}-\frac{x}{A\v}\Big{)}\cdot\s(-B\v)\\
&=&A^{-2}\v^{-2}\Big{(}\frac{\sqrt{3}\v^{3}}{3xy}+x\Big{)}^2\cdot\Big{(}\frac{\sqrt{3}\v^{3}}{3xy}-x\Big{)}\cdot\s(-B\v),
\end{eqnarray*}
and the claim is obvious, since $\v^{6}=x^3y^2(3x-2y)$. Moreover, expand $\s(-Bu)$ using the Taylor series given by (\ref{sp}). Then the Taylor coefficients of the last expression are either rational numbers, or are of the form $r\sqrt{3}$, $r\in\mathbb{Q}$. This is clear, since the Taylor series of $\s(-B\v)$ starts at $-B^{2}\v^{2}$, and $\frac{B^2}{A^{2}}$ is a rational number. Moreover, and this is the most important thing when we talk about ramification, the coefficients at $\v^{6n}$ are rational numbers, and the  coefficients at $\v^{6n+3}$ are of the form $\sqrt{3}r$, $r\in\mathbb{Q}$. We thus inspect all the terms in (\ref{alpha}) and arrive at exactly the same conclusion. it is important that if $S_{1},C_{1},S_{2},C_{2}$ are given weight $1$, then the weights in the two numerators and denominators in (\ref{alpha}) are all equal modulo $3$. The
 formula (\ref{alpha}) is completely analogous to the formula presented in Theorem 2, only we do not write the lengthy addition formulas explicitly. As mentioned before, the presented formula for $\alpha$ potentially contain $\v^{6n+3}=[x^3y^2(3x-2y)]^{(2n+1)/2}$ and hence a ramification, since $p(\v)$ and $q(\v)$ contain this. However, the coefficients at $\v^{6n+3}$ are necessarily of the form $r\sqrt{3}$, $r\in\mathbb{Q}$. But we know that $\alpha$, if written as a function in $x,y$ and $\v$, is an even function in $\v$. This is clear from the fact that the exact value for $\alpha$, which is unique, does not change if we assume that the sign $``\sqrt{3}"$ stands for a negative value of the square root rather than a positive. The same conclusion follows for the function $\beta$. And so, the flow $\alpha\bl\beta$ is indeed unramified! This finishes the case 6) of Theorem 1.\\

We note that the constant $\pi_{3}^{6}$ on p. 883 was miscalculated by the factor $27$. The correct value is
\begin{eqnarray*}
\pi_{3}^{6}=820.824437079556_{+}\cdot 27=22162.259801148018_{+}.
\end{eqnarray*}


\begin{thebibliography}{9}

\bibitem{aczel}{\sc J. Acz\'{e}l}, {\it Lectures on functional equations and their applications}, Mathematics in Science and Engineering, Vol. 19 Academic Press, New York-London 1966.

\bibitem{alkauskas-t} {\sc G. Alkauskas}, Multi-variable translation equation which arises from homothety. {\it Aequationes Math.}
{\bf 80}(3) (2010), 335--350.

\bibitem{alkauskas} {\sc G. Alkauskas}, The projective translation equation and rational plane flows. I. {\it Aequationes Math.}
{\bf 85}(3) (2013), 273--328. 

\bibitem{alkauskas-un} {\sc G. Alkauskas}, The projective translation equation and unramified 2-dimensional flows with rational vector fields.  {\it Aequationes Math.}
{\bf 89}(3) (2015), 873--913. 

\bibitem{alkauskas-comm}{\sc G. Alkauskas}, Commutative projective flows, \url{http://arxiv.org/abs/1507.07457}.

\bibitem{beukers}{\sc F. Beukers},
Hypergeometric functions, how special are they?
{\it Notices Amer. Math. Soc. } {\bf 61} (1) (2014), 48–-56. 

\bibitem{flajolet-b}{\sc R. Bacher, Ph. Flajolet}, Pseudo-factorials, elliptic functions, and continued fractions. {\it Ramanujan J.}
{\bf 21} (1)(2010), 71--97.

\bibitem{flajolet}{\sc Ph. Flajolet, J. Gabarr\'{o}, H. Pekari}, Analytic urns. {\it Ann. Probab.} {\bf 33}(3) (2005), 1200--1233; \url{arxiv.org/abs/math/0407098}.

\bibitem{flajolet-c}{\sc E.van Fossen Conrad, Ph. Flajolet}, The Fermat cubic, elliptic functions, continued fractions, and a combinatorial excursion.
{\it S\'{e}m. Lothar. Combin.} {\bf 54} (2005/07), Art. B54g; \url{arxiv.org/abs/math/0507268}.

\bibitem{lang}{\sc S. Lang}, {\it Introduction to algebraic and abelian functions}, Second edition. Graduate Texts in Mathematics, 89. Springer-Verlag, New York-Berlin, 1982.

\bibitem{milne}{\sc J. Milne}, {\it Algebraic geometry}, v6.00, \url{http://www.jmilne.org/math/CourseNotes/ag.html}.

\bibitem{moszner1}{\sc Z. Moszner}, The translation equation and its application. {\it Demonstratio Math.} {\bf 6} (1973), 309--327.

\bibitem{moszner2}{\sc Z. Moszner}, General theory of the translation equation. {\it Aequationes Math.}  {\bf 50}(1-2)  (1995), 17--37.

\bibitem{mumford}{\sc D. Mumford}, {\it Tata lectures on theta. I,  II, III}, Progress in Mathematics, 28. Birkh\"{a}user 2007. 


\bibitem{nikolaev}{\sc I. Nikolaev, E. Zhuzhoma}, {\it Flows on $2$-dimensional manifolds}. An overview.  Lecture Notes in Mathematics, 1705. Springer-Verlag, Berlin (1999).

\bibitem{oeis}{The Online Encyclopedia of Integer Sequences}, \url{https://oeis.org/}, A092676, A092677.

\bibitem{vidunas}{\sc R. Vid\={u}nas}, Darboux evaluations of algebraic Gauss hypergeometric functions. {\it Kyushu J. Math.} {\bf 67} (2)  (2013), 249--280.

\end{thebibliography}
\end{document}